\documentclass[11pt]{amsart}
\usepackage{amssymb}
\usepackage{amsmath,amscd}
\usepackage{combelow}
\usepackage{mathtools}
\usepackage{xcolor}
\usepackage{hyperref}
\usepackage{geometry}
 \usepackage{enumitem}

\newcommand{\vv}{\vert \vert}

\newcommand{\leftv}{\left\vert}
\newcommand{\rightv}{\right\vert}
\newcommand{\vvh}{\vert \vert_{\operatorname{Hof}}}

\newcommand{\R}{\mathbb{R}}

\newcommand{\N}{\mathbb{N}}
\newcommand{\Z}{\mathbb{Z}}
\newcommand{\C}{\mathbb{C}}
\newcommand{\CP}{\mathbb{CP}}

\newcommand{\supp}{\operatorname{supp}}

\newcommand{\Homeo}{\operatorname{Homeo}}
\newcommand{\Ham}{\operatorname{Ham}}
\newcommand{\tHam}{\widetilde{\operatorname{Ham}}}
\newcommand{\bHam}{\overline{\operatorname{Ham}}}
\newcommand{\Id}{\operatorname{Id}}

\newcommand{\Spec}{\operatorname{Spec}}
\newcommand{\Sym}{\operatorname{Sym}}
\newcommand{\Ker}{\operatorname{Ker}}
\newcommand{\im}{\operatorname{Im}}
\newcommand{\osc}{\operatorname{osc}}
\newcommand{\Cal}{\operatorname{Cal}}

\newcommand{\Hameo}{\operatorname{Hameo}}
\newcommand{\FHomeo}{\operatorname{FHomeo}}
\newcommand{\diam}{\operatorname{diam}}
\newcommand{\Area}{\operatorname{Area}}

\newtheorem{theorem}{Theorem}

\newtheorem*{question*}{Question}
\newtheorem{definition}[theorem]{Definition}
\newtheorem{lemma}[theorem]{Lemma}
\newtheorem{prop}[theorem]{Proposition}

\newtheorem*{lemma*}{Lemma}
\newtheorem*{theorem*}{Theorem}
\newtheorem*{remark*}{Remark}
\newtheorem*{definition*}{Definition}
\newtheorem{remark}[theorem]{Remark}
\theoremstyle{remark}

\newtheorem*{remarks*}{Remarks}

\theoremstyle{definition}

\newtheorem*{claim*}{Claim}

\newtheorem*{example*}{Examples}

\title{Hameomorphism groups of positive genus surfaces}
\author{Cheuk Yu Mak and Ibrahim Trifa}

\begin{document}

\maketitle

\begin{abstract}
    In their previous works \cite{CGHMSS1,CGHMSS2}, Cristofaro-Gardiner, Humilière, Mak, Seyfaddini and Smith defined links spectral invariants on connected compact surfaces and used them to show various results on the algebraic structure of the group of area-preserving homeomorphisms of surfaces,  particularly in cases where the surfaces have genus zero.  We show that on surfaces with higher genus, for a certain class of links, the invariants will satisfy a local quasimorphism property. Subsequently, we generalize their results to surfaces of any genus. This extension includes the non-simplicity of (i) the group of hameomorphisms of a closed surface, and (ii) the kernel of the Calabi homomorphism inside the group of hameomorphisms of a surface with non-empty boundary. Moreover, we prove that the Calabi homomorphism extends (non-canonically) to the $C^0$-closure of the set of Hamiltonian diffeomorphisms of any surface. The local quasimorphism property is a consequence of a quantitative Künneth formula for a connected sum in Heegaard Floer homology, inspired by results of Ozsv{\'{a}}th and Szab{\'{o}}.

\end{abstract}

\tableofcontents

\section{Introduction}

Let $\Sigma$ be a compact connected orientable surface (possibly with boundary) equipped with an area from $\omega$.
In 1980's, Fathi defined the mass-flow homomorphism \cite{Fa80}
\[
\Homeo_c(\Sigma,\omega) \to \R
\]
from the group of area-preserving homeomorphisms supported in the interior of $\Sigma$ to $\R$.
Whether its kernel is a simple group was an open question for a long time and has recently been resolved negatively using techniques from symplectic geometry.
The case of the sphere was answered by \cite{CGHS20} using periodic Floer homology, building on the work of \cite{Hutchings} and \cite{CGHR}. The case of positive genus surfaces was answered
by \cite{CGHMSS1} using Lagrangian Floer theory, borrowing ideas from \cite{OS}, \cite{MS21} and \cite{PS21}.

Symplectic geometry enters the picture because $\omega$ is a symplectic form and the kernel of the mass-flow homomorphism can be identified with the $C^0$ closure of the group $\Ham(\Sigma)$ of Hamiltonian diffeomorphisms supported in the interior of $\Sigma$.
The Hofer metric, a bi-invariant and non-degenerate metric, on $\Ham(\Sigma)$ enables us to define two natural normal subgroups of $\overline \Ham(\Sigma)$, namely the group of hameomorphisms $\Hameo(\Sigma)$ and the group of finite energy homeomorphisms $\FHomeo(\Sigma)$ (see Section \ref{ss:subgroups} for the precise definitions, and also \cite{OM}, \cite{CGHS20} for more discussions). 
Indeed, the authors of \cite{CGHMSS1} show that the subgroup $\Hameo(\Sigma)$ is always a proper normal subgroup.

Since then, the method has been pushed further to answer more refined questions about the algebraic structure of $\overline \Ham(\Sigma)$, especially when $\Sigma$ has genus $0$, using a property called the {\it quasimorphism} property.
The goal of this paper is to generalize the results of \cite{CGHMSS2} to all surfaces even though we no longer have the quasimorphism property for positive genus surfaces.


\subsection{Link spectral invariants and known results for genus zero surfaces}

Link spectral invariants are introduced in \cite{CGHMSS1} as the main tool to study $\overline \Ham(\Sigma)$. Given a Lagrangian link (i.e. a union of disjoint circles) $\underline L = L_1\cup ... \cup L_k$ satisfying certain monotonicity conditions on a surface $(\Sigma,\omega)$, we can associate a spectral invariant $c_{\underline L} : C^\infty(S^1\times\Sigma,\mathbb R) \rightarrow \mathbb R$ which satisfies several useful properties (Proposition \ref{prop:invariants}).

In particular, the homotopy invariance permits to define $c_{\underline L}(\varphi)$ for $\varphi \in \widetilde{\Ham}(\Sigma)$, by the formula $c_{\underline L}(\{\varphi_H^t\}_{t \in [0,1]}) = c_{\underline L}(H)$ for a mean normalized Hamiltonian function $H$.
The  homogenization $\mu_{\underline L}$ of $c_{\underline L}$ is defined by 
\[
\mu_{\underline L}(\varphi)= \lim_{n \to \infty} \frac{c_{\underline L}(\varphi^n)}{n}.
\]

In the case of $\Sigma = S^2$, we have the following :

\begin{theorem}[Theorem 7.7 of \cite{CGHMSS1}]
\label{thm:qm}
$c_{\underline L} : \widetilde{\Ham}(S^2)\to\mathbb R$ is a quasimorphism with defect $D\leqslant \frac{k+1}{k}\lambda$ where $\lambda$ is the monotonicity constant of $\underline L$.
Moreover, $\mu_{\underline L}$ descends to a quasimorphism on $\Ham(S^2)$.
\end{theorem}

The fact that $\mu_{\underline L}$ is a quasimorphism and that we can quantify its defect is the key ingredient to prove the following results (the definition of $\Cal$ will be recalled in Section \ref{ss:subgroups}):

\begin{enumerate}
\item 
\label{thm:extension}
The Calabi homomorphism $\Cal:\Hameo(D^2) \to \mathbb R$ can be extended to $\overline{\Ham}(D^2) = \Homeo(D^2,\omega)$. (\cite[Theorem 1.9]{CGHMSS2})

\item
\label{thm:notsimple}
$\Ker(\Cal)\cap \Hameo(D^2)$ is not simple. (\cite[Theorem 1.3]{CGHMSS2})


\item
    \label{thm:hameo}
    $\Hameo(S^2)$ is not simple. (\cite[Theorem 1.3]{CGHMSS2})

\end{enumerate}

\subsection{Main results for positive genus surfaces}

The purpose of this paper is to generalize (\ref{thm:extension}) and (\ref{thm:notsimple}) to any compact oriented surface $\Sigma$ (of any genus) with non-empty boundary:
\begin{theorem}
\label{thm:extensiongeneral}
The Calabi homomorphism $\Hameo(\Sigma) \to \mathbb R$ can be extended to $\overline{\Ham}(\Sigma)$.
\end{theorem}
\begin{theorem}
\label{thm:notsimplegeneral}
$\Ker(\Cal)\cap \Hameo(\Sigma)$ is not simple.
\end{theorem}
and generalize (\ref{thm:hameo}) to any connected closed oriented surface $(\Sigma,\omega)$:

\begin{theorem}
    \label{thm:hameogeneral}
    $\Hameo(\Sigma,\omega)$ is not simple.
\end{theorem}
Theorem \ref{thm:notsimplegeneral} and \ref{thm:hameogeneral} together answer a question in \cite[Problem 4]{OM} for all surfaces.

There is a fundamental difference between the genus $0$ and positive genus case: $c_{\underline L}$ and $\mu_{\underline L}$ are never  quasimorphisms for positive genus surfaces for any $\underline L$ (cf. Proposition \ref{p:notqm}). 
To remedy this, we need to prove a local version of the quasimorphism property when $\Sigma$ has positive genus and combine it with the fragmentation technique.
This requires a slightly different class of Lagrangian links (see Definition \ref{d:admissible}) than those in \cite{CGHMSS1}.
We define the spectral invariants $c_{\underline L}$ for this new class of links, show that they satisfy all the usual spectral invariant properties listed in Proposition \ref{prop:invariants}, as well as the following {\it local quasimorphism} property.

\begin{theorem}
\label{thm:localqm}
    Let $\underline L$ be an admissible link with $k$ contractible components, with monotonicity constant $\lambda$ (see Definition \ref{d:admissible}).
    Let $D\subset \Sigma$ be a disk that does not intersect the non-contractible components of $\underline L$, and denote by $\Ham_D(\Sigma)$ the Hamiltonian diffeomorphisms supported in $D$.
    Then, the restriction of $c_{\underline L}$ to $\Ham_D(\Sigma)$ is a quasimorphism with defect bounded by $\frac{k+1}{k+g}\lambda$.
\end{theorem}

The construction of $c_{\underline L}$ and the proof of its local quasimorphism property relies on the following Künneth formula for connected sums in Heegaard Floer Homology, similar to the stabilization result of \cite{OS}, which is proved by identifying moduli spaces of holomorphic maps under degeneration:
\begin{theorem}
    \label{thm:stabilization}
    Consider two transverse $\eta$-monotone admissible Lagrangian links $\underline L$ and $\underline K$ with $k$ components on a closed surface $(\Sigma,\omega)$. Let $(E,\omega_E)$ denote the two-dimension torus, and $\alpha$ be a non-contractible circle on $E$. Let $\alpha '$ be a small Hamiltonian deformation of $\alpha$, such that $\alpha$ and $\alpha '$ are transverse.
    Then for an appropriate choice of almost complex structure, there is an isomorphism of filtered chain complexes
    \[ CF^*(\Sym \underline L, \Sym\underline K)\otimes CF^*(\alpha,\alpha ') \xrightarrow{\sim} CF^*(\Sym (\underline L\cup \alpha),\Sym (\underline K\cup\alpha '))\]
    where the LHS is computed considering the links $\underline L$ and $\underline K$ in $(\Sigma, \omega)$, $\alpha$ and $\alpha '$ in $(E,\omega_E)$, while in the RHS, $\underline L\cup\alpha$ and $\underline K\cup \alpha '$ are links in the connected sum $(\Sigma \# E,\omega')$ (where we perform the connected sum between a point $\sigma_1 \in \Sigma$ away from the links $\underline L$ and $\underline K$, and a point $\sigma_2 \in E$ away from the isotopy between $\alpha$ and $\alpha '$).
\end{theorem}



If we forget the filtration, Thereom \ref{thm:stabilization} is an identification of generators and differentials so it doesn't depend on the symplectic form. To guarantee that the filtration also agrees, the symplectic form $\omega'$ on $\Sigma \# E$ is chosen such that it equals to $\omega$ away from a neighborhood $B(\sigma_1)$ of $\sigma_1$ which does not intersect $\underline L \cup \underline K$, equals to $\omega_E$ over the support $K_{\alpha}$ of the Hamiltonian isotopy from $\alpha$ to $\alpha'$, and satisfies $\omega'(\Sigma \# E)=\omega(\Sigma)$ (so we need to assume that $\omega_E(K_{\alpha}) < \omega(B(\sigma_1))$ for $\omega'$ to exist).

\subsection*{Structure of the paper}

We collect some preliminaries in Section \ref{s:pre}. 
The new class of Lagrangian links and the proof of its local quasimorphism property (Theorem \ref{thm:localqm}) are given in Section \ref{sec:construction}.
Section \ref{sec:proof} is devoted to the proof of the main results, Theorem \ref{thm:extensiongeneral}, \ref{thm:notsimplegeneral} and \ref{thm:hameogeneral}.
Theorem \ref{thm:stabilization} is proved in Section \ref{sec:stabilization}.

\subsection*{Acknowledgement}

The authors thank Sobhan Seyfaddini for suggesting the exploration of a local quasimorphism property. 
They also thank Ivan Smith for discussions regarding the Floer theory of symmetric products, Kristen Hendricks for discussions concerning Heegaard Floer homology, and Vincent Humili\`ere for discussions about properties of spectral invariants.
C.M. was supported by the Royal Society University Research Fellowship while working on this project.
I.T. was supported by the École Normale Supérieure while working on this project.
I.T. was also partially supported by the ERC Starting number 851701.

\section{Preliminaries}\label{s:pre}

\subsection{Subgroups of $\bHam(\Sigma)$}\label{ss:subgroups}

Let $\Sigma$ be a compact connected surface equipped with an area form $\omega$.
We start by introducing some conventions and notations, which we follow closely from \cite{CGHMSS1}:

\begin{itemize}
    \item Given a Hamiltonian $H:S^1\times \Sigma\to\R$, the Hamiltonian diffeomorphism $\phi^1_H$ is the time 1 flow of the Hamiltonian vector field $X_{H_t}$ defined by $\iota_{X_{H_t}}\omega = dH_t$; 
    \item Given two Hamiltonians $H$ and $K$, we define the composition by $(H \# K)_t(x):=H_t(x)+K_t((\phi^{t}_H)^{-1}(x))$; 
    \item We denote by $\Ham(\Sigma)$ the group of Hamiltonian diffeomorphisms of $\Sigma$ supported in the interior of $\Sigma$ (it is often denoted $\Ham_c(\Sigma)$ in the literature);
    \item $\bHam(\Sigma)$ its closure for the $C^0$ distance;
    \item the Hofer norm of a Hamiltonian is $\vv H_t\vvh := \int_0^1\osc H_t dt = \int_0^1(\max H_t -\min H_t)dt$;
    \item the Hofer norm of a Hamiltonian diffeomorphism is $\vv \varphi\vvh := \inf\limits_{H_t,\varphi = \varphi^1_{H_t}}\vv H_t\vvh$;
    \item the Hofer distance on $\Ham(\Sigma)$ is $d_H(\varphi,\psi):=\vv \varphi\psi^{-1}\vvh$;
\end{itemize}

We define some subgroups of $\bHam(\Sigma)$ (cf. \cite{OM} and \cite{CGHS20}):

\begin{definition}$\varphi\in\overline{\Ham}(\Sigma,\omega)$ is called a \textbf{finite energy homeomorphism} if there exists a sequence of smooth Hamiltonians $H_i$ such that :
\begin{itemize}
\item $\phi^1_{H_i}\xrightarrow{C^0}\varphi$
\item There exists $C\geqslant 0$ such that for every $i$,
\[\vv H_i\vvh := \int_0^1\osc(H_{i,t})dt\leqslant C\]
\end{itemize}
\end{definition}

\begin{definition}
$\varphi\in\overline{\Ham}(\Sigma,\omega)$ is called a \textbf{hameomorphism} if there exists an isotopy $(\psi^t)_{t \in [0,1]}$ in $\bHam(\Sigma)$ from $\Id$ to $\varphi$ and a sequence of smooth Hamiltonians $H_i$ supported in a compact subset $K$ of the interior of $\Sigma$  such that :
\begin{itemize}
    \item $\phi^t_{H_i}\xrightarrow{C^0}\psi^t$ uniformly in $t \in [0,1]$;
    \item $(H_i)$ is a Cauchy sequence for the Hofer norm.
\end{itemize}
\end{definition}

We denote the group of finite energy homeomorphisms by $\FHomeo(\Sigma,\omega)$, and the group of hameomorphisms by $\Hameo(\Sigma,\omega)$.
When $\Sigma$ has non-empty boundary, one can define the Calabi invariant $\Cal : \Ham(\Sigma)\to\R$ as follow:
let $\varphi\in\Ham(\Sigma)$, and $H_t$ be a Hamiltonian supported in the interior of $\Sigma$ such that $\varphi=\phi^1_{H_t}$. Then,
\[\Cal(\varphi)=\int_0^1\int_\Sigma H_t\omega dt\]

This definition does not depend on the choice of the Hamiltonian $H_t$, and $\Cal$ is a group homomorphism.

As shown in \cite{CGHMSS1}, $\Cal$ can be extended canonically to a group homomorphism $\Hameo(\Sigma)\to R$ by the formula $\Cal(\varphi)=\lim\limits_{i\to\infty}\Cal(\varphi^1_{H_i})$, where we consider any sequence $(H_i)$ as in the definition of a hameomorphism.

The purpose of this paper is to study the algebraic structure of $\bHam(\Sigma)$ and its subgroups, for a general surface $\Sigma$.

Here is what was known before this paper:

\begin{enumerate}
    \item $\bHam(\Sigma)$ is not simple since $\FHomeo(\Sigma)$ is a proper normal subgroup (\cite{CGHMSS1});
    \item $\Hameo(S^2)$ is not simple (\cite{CGHMSS2});
    \item $\FHomeo(S^2)$ is not simple since $\Hameo(S^2)$ is a proper normal subgroup (\cite{Buhovsky});
    \item when $\Sigma$ has non-empty boundary :
    \begin{enumerate}
        \item $\Hameo(\Sigma)$ is not simple since it contains the kernel of the (extended) Calabi homomorphism (\cite{CGHMSS1});
        \item $\FHomeo(\Sigma)$ is not simple, since either $\Hameo(\Sigma)$ is a proper normal subgroup, or they coincide and by the previous point they are not simple (\cite{CGHMSS1});
        \item $\Hameo(D^2)\cap\Ker(\Cal)$ is not simple (\cite{CGHMSS2})
    \end{enumerate}
    \item All normal subgroups of $\bHam(\Sigma)$ contain the commutator subgroup, which is perfect and simple.    
\end{enumerate}

We will extend this picture with a generalization of $(2)$, $(3)$ and $(4)(c)$ respectively:
\begin{itemize}
    \item when $\Sigma$ is closed, $\Hameo(\Sigma)$ is not simple (Theorem \ref{thm:hameogeneral});
    \item when $\Sigma$ is closed, $\FHomeo(\Sigma)$ is not simple, since either $\Hameo(\Sigma)$ is a proper normal subgroup, or they coincide and by the previous point they are not simple;
    \item when $\Sigma$ has non-empty boundary, $\Hameo(\Sigma)\cap\Ker(\Cal)$ is not simple (Theorem \ref{thm:notsimplegeneral}).
\end{itemize}

\subsection{Spectral invariants and Quasimorphisms}

Let $(M,\omega)$ be a closed symplectic manifold, and $L\subset M$ a monotone Lagrangian, i.e. $\omega|_{\pi_2(M,L)}=\tau\mu|_{\pi_2(M,L)}$ for some constant $\tau>0$, where $\mu$ is the Maslov homomorphism.
Then, by \cite{LZ}, for a Lagrangian $L'$ Hamiltonian isotopic to $L$, and a Hamiltonian $H$ such that $\varphi^1_H(L) \pitchfork L'$, the Floer cohomology $HF^*(L,L',H)$ is well defined.

We follow the convention in \cite[Section 6]{CGHMSS1} and define the Floer cohomology $HF^*(L,L',H)$ is a vector space over $\C[[T]][T^{-1}]$.
In particular, there is an action filtration on the Floer complex, by defining $CF_\lambda(L,L',H)$, the subcomplex of $CF(L,L',H)$ generated by capped Hamiltonian chords of action less than or equal to $\lambda$. 
The inclusion of this subcomplex gives rise to a map 
\[
i_\lambda : HF_\lambda(L,L',H)\to HF(L,L',H).
\]

We assume that either $L'=L$ or $L' \pitchfork L$. In the former case, there is the PSS isomorphism $QH(L) \to HF(L,L,H)$.
In the latter case, there is the continuation isomorphism $HF(L,L',0) \to HF(L,L',H)$.
By an abuse of notation, we denote $QH(L)$ by $HF(L,L,0)$ and the isomorphism (in either case) by $\kappa$.
Given a homology class $a\in HF^*(L,L',0)\setminus\{0\}$, one can define a spectral invariant:
\[c_{L,L'}(a,H):=\inf\{\lambda | \kappa(a)\in \im i_\lambda\}\]
When $L=L'$ and $a = e_L$ is the unit of $QH^*(L)$, we will simply denote $c_L(H):=c_{L,L}(e_L,H)$.
This spectral invariant satisfies a homotopy invariance property, which enables us to define $c_L$ on $\tHam(M,\omega)$, the universal cover of $\Ham(M,\omega)$.

We recall the definition of a quasimorphism:
\begin{definition}
    Let $G$ be a group. A \textbf{quasimorphism} on $G$ is a map $\mu : G\to \R$ that satisfies:
    \[\exists D\geqslant 0, \forall g,h\in G, |\mu(gh)-\mu(g)-\mu(h)|\leqslant D\]
    The infimal value of $D$ such that this property holds is called the defect of $\mu$.
    
    Moreover, $\mu$ is an \textbf{homogeneous quasimorphism} if it also satisfies
    \[\forall n\in\Z,\forall g \in G,\mu(g^n)=n\mu(g)\]
\end{definition}

When $(M,\omega)=(\CP^n,\omega_{FS})$ and $L$ is a monotone Lagrangian submanifold with $HF(L) \neq 0$, $c_L$ is a quasimorphism on $\tHam(M,\omega)$.
This is a consequence of the same result for the Hamiltonian spectral invariant $c$ (cf. \cite{EP}), and the inequality $c_L\leqslant c$ (cf. \cite[Proposition 4]{LZ}).

\begin{prop}[Homogenization] Let $\mu : G \to \mathbb R$ be a quasimorphism. Then,
\[\widetilde \mu (g):= \lim\limits_{n\to\infty} \frac{\mu(g^n)}{n}\]
is well defined, and it is a homogeneous quasimorphism, called the \textit{homogenization} of $\mu$.
\end{prop}

Now we explain the construction of spectral invariants for Lagrangian links as defined in \cite{CGHMSS1}.

Consider a closed symplectic surface $(\Sigma,\omega)$, with a compatible complex structure $j$.
A Lagrangian link on $\Sigma$ is a disjoint union $\underline L=L_1\cup...\cup L_k$ of smooth simple curves in $\Sigma$.

\begin{definition}
    \label{def:monotone}
    Denote by $B_j$, $1\leqslant j \leqslant s$, the connected components of $\Sigma\setminus\underline L$. Let $k_j$ be the number of boundary components of $B_j$, and $A_j$ the $\omega$-area of $B_j$. Let $\eta \geqslant 0$.
    We say that $\underline L$ is \textbf{$\eta$-monotone} if 
    \[\lambda:=2\eta(k_j-1)+A_j\]
    does not depend on $j$. $\lambda$ is called the \textbf{monotonicity constant} of $\underline L$.
    
    A Lagrangian link $\underline L$ on a compact surface $\Sigma_0$ with non-empty boundary is called $\eta$-monotone if there exists a symplectic embedding of $\Sigma_0$ into a closed surface $\Sigma$ such that $\underline L$ is $\eta$-monotone inside $\Sigma$.
\end{definition}

\begin{remark}
\label{rk:lambda}
    $\lambda$ is equal to the area of the disks bounded by contractible components of the link. Therefore, if $\underline L$ has $m$ components bounding pairwise disjoint disks, then $\lambda\leqslant \frac 1 m$.
\end{remark}

Let $\underline L=L_1\cup...\cup L_k$ be a Lagrangian link on $\Sigma$.
Denote by $\Sym \underline L$ the image of $L_1\times ... \times L_k$ in the symmetric product $\Sym^k(\Sigma) := \Sigma^k/\mathfrak S_k$, where $\mathfrak S_k$ is the permutation group permuting the factors.
Suppose that $\underline L$ is $\eta$-monotone and $\underline L'$ is Hamiltonian isotopic to $\underline L$.
Let $H:S^1 \times \Sigma \to \R$ be a Hamiltonian and $\Sym^k(H):S^1 \times \Sym^k(\Sigma) \to \R$ be given by 
$\Sym^k(H)_t(x_1,\dots,x_k):=\sum_{i=1}^k H_t(x_i)$.
We recall in Section \ref{sec:Heegaard} how from such a link one can define a Floer cohomology\footnote{The function $\Sym^k(H)$ is not smooth along the diagonal of $\Sym^k(\Sigma)$ but it turns out that any smooth Hamiltonian that agrees with $\Sym^k(H)$ outside a sufficiently small neighborhood of the diagonal will give the same Floer cohomology up to canonical isomorphisms as a filtered vector space. Therefore, $HF^*(\Sym\underline L,\Sym\underline L',\Sym^k(H))$ is defined to be the filtered vector space.} 
\begin{align}\label{eq:L=SymL}
HF(\underline L,\underline L',H):= HF^*(\Sym\underline L,\Sym\underline L',\Sym^k(H)).
\end{align}
It was shown in \cite{CGHMSS1} that $HF^*(\Sym\underline L,\Sym\underline L',\Sym^k(H))\cong H^*(\Sym \underline L)$ so a vector space (without filtration) so it is non-zero. 
Moreover, they show that Lagrangian spectral invariants
$c_{\Sym\underline L,\Sym\underline L'}(a,\Sym^k(H))$ are well-defined.
Therefore, one can define  link spectral invariants
\[c_{\underline L} := \frac 1 k c_{\Sym \underline L}=\frac 1 k c_{\Sym \underline L, \Sym \underline L}(e_{\Sym \underline L}, \cdot)\]

\begin{prop}
\label{prop:invariants}
This invariant inherits all the properties of Lagrangian spectral invariants:
\begin{itemize}\item (spectrality) $c_{\underline L}(H)$ lies in the action spectrum $\Spec(H,\underline L)$
\item (Hofer Lipschitz) $\left|c_{\underline L}(H)-c_{\underline L}(K)\right|\leqslant\vv H-K\vvh$
\item (monotonicity) If $H\leqslant K$ then $c_{\underline L}(H)\leqslant c_{\underline L}(K)$
\item (Lagrangian control) If $H_t|_{L_i}=s_i(t)$ for each $i$, then
\[c_{\underline L}(H)=\frac{1}{k}\sum\limits_{i=1}^k \int s_i(t)dt\]
Moreover,
\[\frac 1 k \sum\limits_{i=1}^k\int_{S^1}\min\limits_{L_i}H_t dt\leqslant c_{\underline L}(H)\leqslant \frac 1 k \sum\limits_{i=1}^k\int_{S^1}\max\limits_{L_i}H_t dt \]
\item (triangle inequality) $c_{\underline L}(H\# K)\leqslant c_{\underline L}(H)+c_{\underline L}(K)$
\item (homotopy invariance) If $H,K$ are mean normalized, $\phi^1_H=\phi^1_K$ and $(\phi^t_H)_{t \in [0,1]}$ is homotopic to $(\phi^t_K)_{t \in [0,1]}$ relative to endpoints, then $c_{\underline L}(H)=c_{\underline L}(K)$.
\item (shift) $c_{\underline L}(H+s(t))=c_{\underline L}(H)+\int s(t)dt$
\end{itemize}
\end{prop}

The homotopy invariance permits to define $c_{\underline L}(\{\varphi^t\}_{t \in [0,1]})$ for $\{\varphi^t\}_{t \in [0,1]} \in \widetilde{\Ham}(\Sigma)$, by the formula $c_{\underline L}(\{\phi_H^t\}_{t \in [0,1]}) = c_{\underline L}(H)$ for a mean normalized $H$.

Moreover,  $c_{\underline L}$ a quasimorphism when $\Sigma = S^2$ (i.e. Theorem \ref{thm:qm}).
It is proved using the fact that $\Sym^k(S^2)\cong \CP^k$ (cf. \cite{EP}).

\section{Construction of the new invariants}
\label{sec:construction}

Let $(\Sigma,\omega)$ be a compact surface of genus $g$. We suppose that $\Sigma$ has area $1$.
We introduce the following class of links, which is slightly different from the ones in \cite{CGHMSS1} (cf. \cite{Chen21}, \cite{Chen22} for the study of this class of links in the cylindrical setting).

\begin{definition}\label{d:admissible}
    A Lagrangian link $\underline L = L_1\cup ...\cup L_k\cup \alpha_1\cup...\cup \alpha_g$ is called \textbf{admissible} if:
    \begin{itemize}
        \item the circles $L_1, ..., L_k, \alpha_1,...,\alpha_g$ are all disjoint;
        \item $\alpha_1,...,\alpha_g$ are non-contractible;
        \item there exists a decomposition of $\Sigma$ as a connected sum of a genus zero surface $\Sigma_0$ and $g$ tori such that each $\alpha_i$ lives in a different torus and $L_i$ lives in $\Sigma_0$;
        \item $\underline L_0:=L_1\cup...\cup L_k\subset \Sigma_0$ is $\eta$-monotone for some $\eta\geqslant 0$, with respect to a symplectic form $\omega_0$ on $\Sigma_0$ which coincides with $\omega$ outside a small neighborhood of the connected sum region away from the link, and such that $\omega_0(\Sigma_0)=1$.
    \end{itemize}
    We define the monotonicity constant of $\underline L$ as the monotonicity constant of $\underline L_0$ (see Definition \ref{def:monotone}).
\end{definition}

\begin{remark}[A remark on the third bullet of Definition \ref{d:admissible}]
\label{rk:admissible}
Suppose that $\underline L = L_1\cup ...\cup L_k\cup \alpha_1\cup...\cup \alpha_g$ satisfies the first two bullets of Definition \ref{d:admissible}.
Let $B$ be the image of $H_1(\partial \Sigma) \to H_1(\Sigma)$, $V$ be the image of  $H_1(\alpha_1 \cup \dots \cup \alpha_g) \to H_1(\Sigma)$ and $l_i$ be the image of  $H_1(L_i) \to H_1(\Sigma)$.
Topologically, if
$V$ is a $g$ dimensional subspace which intersects $B$ only at $0$ and $l_i \subset B$ for all $i$, then there is a decomposition of $\Sigma$ as  a connected sum of a genus zero surface $\Sigma_0$ and $g$ tori such that the third bullet of Definition \ref{d:admissible} is satisfied.

To see this, for simplicity, we first assume that there is no $L_i$ and $\Sigma$ is closed (so $B=0$). Then $V$ is a Lagrangian subspace with respect to the intersection form $\Omega$ on $H_1(\Sigma)$.
Let $a_i:=[\alpha_i] \in H_1(\Sigma)$. 
We can complete $\{a_i\}$ to a basis $\{a_1, \dots, a_g, b_1, \dots, b_g\}$ of $H_1(\Sigma;\mathbb{Z})$ such that $\Omega(a_i,b_i)=1$ and $\Omega(a_i,b_j)=0$ if $i \neq j$, and $\Omega(b_i,b_j)=0$ for all $i,j$.
We can find circles $\beta_i \subset \Sigma$, $i=1,\dots,g$, such that the geometric intersection number between any two circles in $\{\alpha_i,\beta_j\}$ agrees with the homological intersection number. 
The regular neighborhood of $\alpha_i \cup \beta_i$ gives the splitting of the $i^{th}$ torus in the connected sum decomposition.
The case when $\Sigma$ has boundary components can be proved by first embedding it to a closed surface by capping off the boundary components by disks (and choosing $\beta_i$ to avoid the capping disks).
The case when there is $L_i$ can be reduced to the case with no $L_i$ by running the argument above, for the positive genus components, in the complement of $\cup_i L_i$ (in particular, $L_i$ are allowed to be non-contractible separating circles).

\end{remark}

\begin{figure}[h]
	\centering
	\includegraphics[width=1\linewidth]{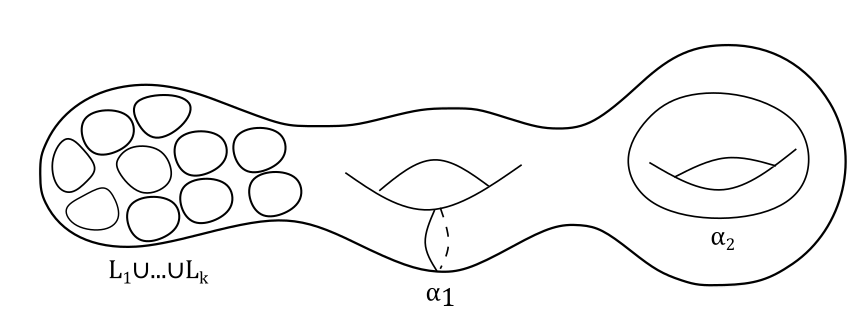}
	\caption{An admissible link}
	\label{fig:Figure 1}
\end{figure}

We assume that $\Sigma$ is closed. Then, given an admissible link $\underline L$, there exists a decomposition of $\Sigma$ as a connected sum $\Sigma = S^2\# E_1 \#...\# E_g$, where the $E_i$ are copies of the 2-torus, such that $\underline L_0:=L_1\cup...\cup L_k \subset S^2$ is $\eta$-monotone, and for all $1\leqslant i \leqslant g$, $\alpha_i \subset E_g$. (Here, we inflate the symplectic form near the connected sum point in $S^2$ so that $S^2$ has area $1$. This choice of symplectic form makes 
$\underline L_0$ $\eta$-monotone by the fourth bullet of Definition \ref{d:admissible}, and it
is compatible with the one in Theorem \ref{thm:stabilization}.)

The authors of \cite{CGHMSS1} show that $HF^*(\underline L_0,\underline L_0)$ is well-defined and non-zero.

By applying Theorem \ref{thm:stabilization} $g$ times, we get that $HF^*(\underline L, \underline L)$ is non-zero, and therefore for a non-degenerate Hamiltonian $H$, $HF^*(\underline L, \underline L,H)$ is also non-zero.
As a result, one can define spectral invariants 
\[c_{\underline L}(H):=\frac 1 {k+g} c_{\Sym\underline L}(\Sym^{k+g}(H))\]
for non-degenerate $H$, and then extend it to all Hamiltonians by continuity (i.e. the Hofer Lipschitz property in Proposition \ref{prop:invariants}).

If $\Sigma_0$ has non-empty boundary, then one can embed $\Sigma_0$ into a closed surface $\Sigma$, such that $\underline L$ remains admissible in $\Sigma$. Indeed, by the definition of $\eta$-monotonicity for surfaces with boundary, there exists an embedding into a closed surface $\Sigma$ such that $\underline L$ is still $\eta$-monotone inside $\Sigma$. 
Then, one can define the link spectral invariant for $\Sigma_0$ by restricting $c_{\underline L}$ to $\Ham(\Sigma_0)\subset \Ham(\Sigma)$.

The fact that this invariant satisfy all the properties listed in Proposition \ref{prop:invariants}, as in \cite{CGHMSS1}, is a straightforward consequence of the properties of Lagrangian spectral invariants (see \cite{H} for instance).

Before proving the local quasimorphism property \ref{thm:localqm}, we show the following statement:

\begin{prop}\label{p:notqm}
Let $\Sigma$ be a surface of genus $g>0$. Let $\underline L = L_1\cup... \cup L_k \cup \alpha_1 \cup ...\cup \alpha_g$ be a monotone admissible Lagrangian link on $\Sigma$, where $\alpha_1$, ..., $\alpha_g$ are the non-contractible components of $\underline L$.
Then $c_{\underline L}$ is not a quasimorphism.
\end{prop}

\begin{proof} It is enough to show that there exists a sequence of Hamiltonians $(H_n)_n$ such that $\gamma_{\underline L}(H_n) := c_{\underline L}(H_n) + c_{\underline L}(\overline{H}_n)$ is not bounded.
We pick a non-contractible circle in $\Sigma$ that intersects $\underline L$ at a single point in $\alpha_1$. Such a circle always exists, take for instance $\beta_1$ as in Remark \ref{rk:admissible}.
Then, we pick a small neighborhood $U$ of this circle, diffeomorphic to the annulus $A=S^1\times (-1;1)$ (we denote by $\psi:U\rightarrow A$ such a diffeomorphism), such that $U\cap\underline L = U\cap \alpha_1$ is connected and sent to a vertical $\lbrace \theta_0\rbrace\times(-1;1)$ by $\psi$.
Let $H:(-1;1)\rightarrow \R$ be a smooth function such that :
\begin{itemize}
    \item $H$ is compactly supported;
    \item $H$ admits a single local maximum at 0 and no other critical point in the interior of its support;
    \item $H(0)=1$
\end{itemize}
We define $K_n$ on $A$ by $K_n(\theta,t) = n H(t)$, and $H_n$ on $\Sigma$ by:
\begin{itemize}
    \item $H_n(x) = K_n(\psi(x))$ if $x\in U$
    \item $H_n(x)=0$ if $x\notin U$
\end{itemize}

Now, we compute the sequence $(\gamma_{\underline L}(H_n))_n$ for this choice of Hamiltonians.

We know that $c_{\underline L}(H_n)$ lies in $\frac 1 {k+g} \Spec(\Sym(H_n))$. In order to compute this spectrum, we consider critical points of the action that are in the same connected component as a chosen reference path in $\mathcal P(\Sym(\underline L),\Sym(\underline L))$ (see Section \ref{sec:Heegaard} for a definition of the Heegaard Floer complex and the action functional).
We pick $x_1\in L_1,...,x_k\in L_k,y_1\in \alpha_1,...,y_g\in\alpha_g$ fixed by the flow of $H_n$, and take the constant path $\eta:=\lbrace x_1,...,x_k,y_1,...,y_g\rbrace$ in $\Sym(\underline L)$ as the reference path.

Then, the only critical points of the action that are in the same connected component as $\eta$ in $\mathcal P(\Sym \underline L, \Sym \underline L)$ are symmetric products of points fixed by $H_n$. They all have zero action except when we choose in $\alpha_1$ the point $y'_1$ for which $H_n$ is maximal. For any choice of $x'_i\in L_i, 1\leqslant i \leqslant k$, and $y'_i\in\alpha_i, 2\leqslant i \leqslant g$, the critical point $\lbrace x'_1,...,x'_k,y'_1, ..., y'_g\rbrace$ has action $n$.

Hence, $\Spec(\Sym(H_n))=\lbrace 0,n\rbrace$. Similarly, $\Spec(\Sym(\overline H_n))=\lbrace -n, 0\rbrace$ because $\overline H_n=-H_n$.
Therefore, $\gamma_{\underline L}(H_n) \in \lbrace -\frac n {k+g}, 0, \frac n {k+g}\rbrace$.

Since $\gamma_{\underline L}$ is non-negative, we can rule out $-\frac n {k+g}$. Moreover, $\Sym(\underline L)$ is not fixed by $\varphi_{H_n}$, so $\gamma_{\underline L}(H_n)$ is non-zero.

Finally, we get that $\gamma_{\underline L}(H_n)=\frac n {k+g}$, which is unbounded as $n$ goes to infinity.
\end{proof}

We now prove that this invariant satisfies Theorem \ref{thm:localqm}.

\begin{proof}[Proof of Theorem \ref{thm:localqm}]
    We consider an admissible link $\underline L = L_1\cup...\cup L_k\cup\alpha_1\cup...\cup\alpha_g$, and a disk $D$ that does not intersect $\underline \alpha := \alpha_1\cup...\cup\alpha_g$. Then, one can find a decomposition of $\Sigma$ as a connected sum $\Sigma = S^2\# E_1 \#...\# E_g$ such that $\underline L_0:=L_1\cup...\cup L_k \subset S^2$ is $\eta$-monotone, for all $1\leqslant i \leqslant g$, $\alpha_i \subset E_g$, and $D\subset S^2$.

    Let $H$ be a Hamiltonian supported in $D$, and let $H_\epsilon$ be an $\epsilon$-perturbation of $H$ in small neighborhoods of the link's components so that $HF(\underline L, \underline L, H_\epsilon)$ is well defined (cf. \eqref{eq:L=SymL}).
    We can assume that $H_{\epsilon}$ is chosen such that it is away from the connected sum neighborhoods of the decomposition
    $\Sigma = S^2\# E_1 \#...\# E_g$.
   
    Then, by applying Theorem \ref{thm:stabilization} $g$ times, we have that
    \begin{align*}
     CF^*(\underline L, \underline L, H_{\epsilon}) 
    \simeq  CF^*(\underline L_0, \underline L_0, H_{\epsilon}|_{S^2})\otimes  \bigotimes_{i=1}^g CF^*(\alpha_i,\alpha_i, H_{\epsilon}|_{E_i})
    \end{align*}

    Therefore, representatives of $\kappa(e_{\Sym(\underline L)})$ in $CF^*(\underline L, \underline L, H_{\epsilon})$ are in one-to-one correspondence with tensor products of representatives of unit classes in $CF^*(\underline L_0, \underline L_0, H_{\epsilon}|_{S^2})$ and $CF^*(\alpha_i,\alpha_i,H_{\epsilon}|_{E_i})$.
It follows from the proof of Theorem \ref{thm:stabilization} that this one-to-one correspondence preserves the action (which is defined as the sum of the actions on the tensor product). 
    
    Since $H_\epsilon$ is $\epsilon$-small on $E_i$, we get that $c_{\Sym\underline L}(\Sym^{k+g}(H)) = c_{\Sym\underline L_0}(\Sym^k(H))$ where $c_{\Sym\underline L_0}$ is computed inside $\Sym^k(S^2)$. Since $L_0$ is $\eta$-monotone inside $S^2$, with monotonicity constant $\lambda$, applying Theorem \ref{thm:qm} gives that the restriction of $c_{\Sym\underline L}$ to $\Ham_D(\Sigma)$ is a quasimorphism with defect bounded by $\frac{k+1}{k+g}\lambda$.
\end{proof}

We define the homogenized spectral invariant $\mu_{\underline L}$ by the formula:
\[\mu_{\underline L}(H):=\lim\limits_{n\to\infty}\frac {c_{\underline L}(H^{\# n})} n\]
This is well defined by the triangle inequality and Fekete's lemma.

\begin{prop}
The invariant $\mu_{\underline L}$ satisfies the following properties:
\begin{itemize}
\item (Hofer Lipschitz) $\left|\mu_{\underline L}(H)-\mu_{\underline L}(K)\right|\leqslant \vv H-K\vvh$ 
\item (Lagrangian control) Suppose $H$ is mean-normalized, and $H_t|_{L_i}=s_i(t)$. Then,
\[\mu_{\underline L}(H)=\frac{1}{k}\sum\limits_{i=1}^k\int_0^1 s_i(t)dt\]
Moreover,
\[\frac{1}{k}\sum\limits_{i=1}^k\int_0^1 \min\limits_{L_i} H_tdt\leqslant \mu_{\underline L}(H)\leqslant \frac{1}{k}\sum\limits_{i=1}^k\int^1_0 \max\limits_{L_i} H_tdt\]
\item (homotopy invariance) $\mu_{\underline L}$ descends to a map $\Ham(\Sigma)\to\R$
\item (support control) If $\supp(\varphi)\subset \Sigma\setminus \underline L$, then $\mu_{\underline L}(\varphi)=-\Cal(\varphi)$.
\item (conjugacy invariance) $\mu_{\underline L}(\psi\varphi\psi^{-1})=\mu_{\underline L}(\varphi)$
\end{itemize}
\end{prop}

\begin{proof}
    These are all straightforward consequences of the properties of $c_{\underline L}$ (\ref{prop:invariants}) and the definition of $\mu_{\underline L}$.    
\end{proof}

Moreover, we show the following:

\begin{theorem}
    \label{thm:mu}
    Suppose that $\underline L$ and $\underline L'$ are two admissible $\eta$-monotone links with the same number of components $k$, that share the same non-contractible components $\underline\alpha$. Then the homogenized spectral invariants $\mu_{\underline L}$ and $\mu_{\underline L'}$ coincide, and we denote by $\mu_{k,\eta,\underline\alpha}$ their common value.
\end{theorem}

\begin{proof}
    Let $*$ denote the pants product 
    \[HF^*(\underline L,\underline L')\otimes HF^*(\underline L', \underline L)\to HF^*(\underline L,\underline L).\]
    Using Theorem \ref{thm:stabilization}, we can view it as a map \[HF^*(\underline L_0,\underline L_0')\otimes HF^*(\underline \alpha,\underline \alpha)\otimes HF^*(\underline L'_0, \underline L_0)\otimes HF^*(\underline \alpha,\underline \alpha)\to HF^*(\underline L_0,\underline L_0)\otimes HF^*(\underline \alpha,\underline \alpha)\]
    Since $\underline L$ and $\underline L_0$ are two $\eta$-monotone links with the same number of components in $S^2$, there exist classes $a_0\in HF^*(\underline L_0,\underline L_0')$ and $b_0\in  HF^*(\underline L'_0, \underline L_0)$ such that $a_0 * b_0=e_{\Sym \underline L_0}\in HF^*(\underline L_0,\underline L_0)$.

    Let $a$ be the image of $a_0\otimes e_{\underline\alpha}$ in $HF^*(\underline L,\underline L')\cong HF^*(\underline L_0,\underline L_0')\otimes HF^*(\underline \alpha,\underline \alpha)$, and $b$ the image of $b_0\otimes e_{\underline\alpha}$ in $HF^*(\underline L',\underline L)\cong HF^*(\underline L_0',\underline L_0)\otimes HF^*(\underline \alpha,\underline \alpha)$.

    Then, $a*b$ is the image of $(a_0*b_0)\otimes (e_{\underline\alpha}*e_{\underline\alpha})=e_{\Sym \underline L_0}\otimes e_{\underline\alpha}$, i.e. $a*b=e_{\Sym \underline L}$ is the unit of $HF^*( \underline L, \underline L)$.

    Then, by the subadditivity property of Lagrangian spectral invariants, we have for any Hamiltonian $H$:
    \begin{align*}
        &c(\Sym \underline L,\Sym\underline L,e_{\Sym\underline L}, H)\\ 
        \leqslant& c(\Sym\underline L,\Sym\underline L',a, H)+c(\Sym\underline L',\Sym\underline L,b,0)\\
        \leqslant& c(\Sym\underline L,\Sym\underline L',a, 0)+c(\Sym\underline L',\Sym\underline L',e_{\Sym\underline L'}, H)+c(\Sym\underline L',\Sym\underline L,b,0)
    \end{align*}

    i.e.
    \[c_{\underline L}(H)\leqslant c_{\underline L'}(H)+\frac 1 k\left(c(\Sym\underline L,\Sym\underline L',a, 0)+c(\Sym\underline L',\Sym\underline L,b,0)\right)\]

    We get for all $n>0$:
    \[\frac {c_{\underline L}(H^{\#n})} n \leqslant \frac {c_{\underline L'}(H^{\#n})} n+ \frac {c(\Sym\underline L,\Sym\underline L',a, 0)+c(\Sym\underline L',\Sym\underline L,b,0)} {kn}\]
    and therefore $\mu_{\underline L}(H)\leqslant \mu_{\underline L'}(H)$.
    Swapping the roles of $\underline L$ and $\underline L'$, we get the other inequality and finally $\mu_{\underline L}=\mu_{\underline L'}$.

\end{proof}

Since the homogenized spectral invariants are conjugacy invariant, $\mu_{k,\eta,\underline \alpha}=\mu_{k,\eta,\underline \alpha'}$ when $\underline\alpha$ and $\underline\alpha'$ are Hamiltonian isotopic, and therefore we can write $\mu_{k,\eta,[\underline \alpha]}:=\mu_{k,\eta,\underline \alpha}$ where $[\underline \alpha]$ is the class of $\underline \alpha$.

We fix a decomposition of $\Sigma$ as a connected sum $\Sigma=\Sigma_0\#E_1\#...\#E_g$ where $\Sigma_0$ is a genus zero surface, and the $E_i$ are copies of the 2-torus. Recall that we modify the symplectic form in a neighborhood of the connected sum points so that $\Sigma_0$ has area 1. 

Let $\beta^1_{j,\theta}$ be the circle $\{\theta\}\times S^1 \subset S^1\times S^1 \cong E_j$, and $\beta^2_{j,\theta}$ be the circle $S^1\times \{\theta\}\subset E_j$.

For $\underline \theta = (\theta_1,...,\theta_g)\in (S^1)^g$, and $\underline \epsilon=(\epsilon_1,...,\epsilon_g)\in\{1,2\}^g$, let $\underline\alpha^{\underline \epsilon}_{\underline \theta}:=\beta^{\epsilon_1}_{1,\theta_1}\cup...\cup\beta^{\epsilon_g}_{g,\theta_g}$. When the components of $\underline\alpha^{\underline \epsilon}_{\underline \theta}$ do not intersect the connected sum regions, this defines a Lagrangian link on $\Sigma$.

\begin{figure}[h]
	\centering
	\includegraphics[width=0.7\linewidth]{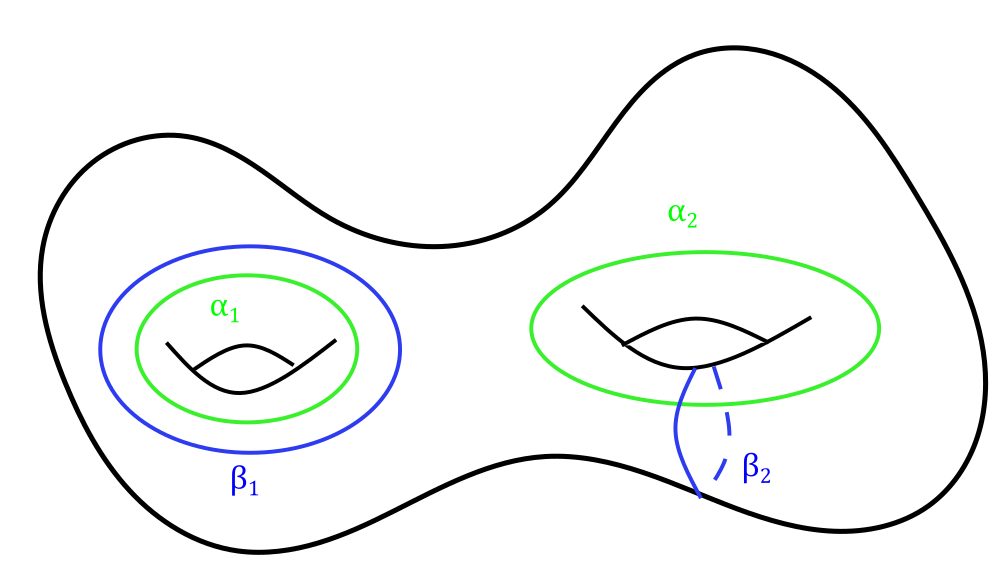}
	\caption{Two links $\underline \alpha$ and $\underline \beta$ inducing independent invariants $\mu_{\underline\alpha}$ and $\mu_{\underline\beta}$ }
	\label{fig:Figure 2}
\end{figure}

\begin{prop}
    The $\{\mu_{k,\eta,[\underline\alpha^{\underline\epsilon}_{\underline \theta}]}\}$ are linearly independent.
\end{prop}

\begin{proof}
    Let $E$ be the vector space generated by the $\{\mu_{k,\eta,[\underline\alpha^{\underline\epsilon}_{\underline \theta}]}\}$, and $E_{k,\eta}$ the vector subspace generated by the $\{\mu_{k,\eta,[\underline\alpha^{\underline\epsilon}_{\underline \theta}]}\}$ where $k$ and $\eta$ are fixed.
    Following the argument in \cite{CGHMSS1}, one can show that $E=\bigoplus\limits_{k,\eta}E_{k,\eta}$.

    Now we show that for fixed $k$ and $\eta$, the $\{\mu_{k,\eta,[\underline\alpha^{\underline\epsilon}_{\underline \theta}]}\}$ are linearly independent.   
    For $1\leqslant j \leqslant g$, let $E_{k,\eta,\beta^\epsilon_{j,\theta}}$ be the subspace generated by the $\mu_{k,\eta,[\underline\alpha^{\underline\epsilon}_{\underline \theta}]}$ that satisfy $\epsilon_{j}=\epsilon$ and $\theta_j=\theta$.
    We are going to show that 
    for every $j=1,\dots,g$, we have 
    \begin{align}\label{eqLdirectsum}E_{k,\eta}=\bigoplus_{\epsilon,\theta} E_{k,\eta,\beta^\epsilon_{j,\theta}}.\end{align}

    Let $l$ and $m$ be non-negative integers. Pick $l$ different elements $\theta_1,...,\theta_l$ in $S^1$, and let $\mu_i$ be an element of $E_{k,\eta,\beta^1_{j,\theta_i}}$.
    We also pick $m$ different elements $\theta_{l+1},...,\theta_{l+m}$ in $S^1$, and let $\mu_{l+i}$ be an element of $E_{k,\eta,\beta^2_{j,\theta_{l+i}}}$.
Let $a_i$ be real numbers such that 
\[\sum\limits_{i=1}^{l+m}a_i\mu_i=0.\] 
We want to show that for all $i$, $a_i=0$.

    Let $V$ be a small neighborhood of $\beta^1_{j,\theta_1}$ that does not intersect the connected sum points and the $\beta^1_{j,\theta_i}$ for $2\leqslant i\leqslant l$. Let $H$ be a Hamiltonian supported in $V$ such that $H|_{\beta^1_{j,\theta_1}}\equiv 1$.

    Let $\theta_0\in S^1-\{\theta_1,...,\theta_l\}$ be such that $\beta^1_{j,\theta_0}$ is away from the connected sum points. For $0\leqslant i\leqslant l$, let $\rho_i$ be the rotation of the torus defined by
    \[\rho_i (\theta,\varphi)=(\theta +\theta_i-\theta_1, \varphi)\]
    We can assume that $V$ is small enough such that for any $i=0,\dots,l$, $\rho_i(V)$ is a  neighborhood of $\beta^1_{j,\theta_i}$ that does not intersect the connected sum points and the $\beta^1_{j,\theta_s}$ for $s \neq i$.
    Let $H_i:=H\circ \rho_i^{-1}$, which is supported in $\rho_i(V)$.

    Then, by the Lagrangian control property, for $0\leqslant i\leqslant l$ and $1\leqslant p \leqslant l$, $\mu_p(H_i)=\frac 1 {k+g} \delta_{i,p}$.
    Moreover, for $l+1\leqslant p \leqslant l+m$, since the $\rho_i$ stabilize the $\beta^2_{j,\theta_{p}}$, we get that $\mu_p(H_i)$ does not depend on $i$.

    Therefore, applying the equality $\sum\limits_{i=1}^{l+m}a_i\mu_i=0$ at $H_i$ for $1\leqslant i\leqslant l$ gives
    \[\frac {a_i} {k+g}=-\sum\limits_{p=l+1}^{l+m}a_p\mu_p(H_1)\]
    and at $H_0$ :
    \[\sum\limits_{p=l+1}^{l+m}a_p\mu_p(H_1)=0\]
    Thus, $a_i = 0$ for $1\leqslant i\leqslant l$. Since the $\beta^1$ and $\beta^2$ play symmetric roles, we can show in the same way that $a_i = 0$ for $l+1\leqslant i\leqslant l+m$.

    Therefore, we have \eqref{eqLdirectsum} for all $j$, and hence $\{\mu_{k,\eta,[\underline\alpha^{\underline\epsilon}_{\underline \theta}]}\}$ are linearly independent.  
    
\end{proof}

\section{Extending the Calabi homomorphism and simplicity}
\label{sec:proof}

The proofs of Theorem \ref{thm:extensiongeneral},  \ref{thm:notsimplegeneral} and \ref{thm:hameogeneral}  will be given in the following three subsections respectively. The main idea is to replace the quasimorphism property (Theorem \ref{thm:qm}), which no longer exists for positive genus surfaces, by Theorem \ref{thm:localqm} and the fragmentation technique. Some of the estimates is a bit more delicate than the ones in \cite{CGHMSS2}.

\subsection{Proof of Theorem \ref{thm:extensiongeneral}}
We can now give a proof of Theorem \ref{thm:extensiongeneral}, inspired by the proof of (\ref{thm:extension}) found in \cite{CGHMSS2}.

\begin{definition}A sequence of admissible Lagrangian links $(\underline L^k)$ is called \textbf{equidistributed} if:
\begin{itemize}
    \item all the $\underline L^k$ share the same contractible components $\alpha_1$,..., $\alpha_g$;
    \item $\underline L^k$ has $k$ contractible components $\underline L^k_1$,..., $\underline L^k_k$;
    \item the $L^k_i$ bound disjoint disks $D^k_i$, and $\diam(\underline L^k):= \max(\diam D^k_i)\xrightarrow[k \to \infty]{} 0$.
\end{itemize}
\end{definition}

Given such a sequence, we get a sequence of link spectral invariants $c_{\underline L^k}$ which satisfies the Calabi property:

\begin{prop}
\label{prop:Calabi}
    For any smooth Hamiltonian $H : S^1\times \Sigma \rightarrow \R$,
    \[c_{\underline L^k}(H)=\int_{S^1}\int_{\Sigma}H_t\omega dt+\operatorname{O}\limits_{k\to\infty}\left(\diam(\underline L^k) \right)\]
    In particular, for any smooth Hamiltonian diffeomorphism $\varphi$,
    \[c_{\underline L^k}(\varphi) = \operatorname O_{k\to\infty}\left(\diam(\underline L^k) \right)\]
\end{prop}

\begin{proof}
We fix a point $x^k_i$ in each of the disks $D^k_i$.
Then, one can find smooth Hamiltonians $G^k$ such that:
\begin{itemize}
    \item $G^k_t\equiv H_t(x^k_i)$ on $D^k_i$
    \item $G^k_t = H_t$ on $\alpha_i$
    \item $\vv G^k_t - H_t\vv_\infty\leqslant \diam(\underline L^k) \sup\limits_{S^1\times\Sigma} \vv dH_t \vv$
\end{itemize}

Then, using the Hofer Lipschitz property and Lagrangian control (Proposition \ref{prop:invariants}), we get:
\begin{align*}
    &\leftv c_{\underline L^k}(H) - \int_{S^1}\int_{\Sigma}H_t\omega dt\rightv \\
    \leqslant& \leftv c_{\underline L^k}(H) - c_{\underline L^k}(G^k)\rightv + \leftv c_{\underline L^k}(G^k)-\int_{S^1}\int_{\Sigma}G^k_t\omega dt \rightv + \leftv\int_{S^1}\int_{\Sigma}(G^k_t-H_t)\omega dt\rightv \\
    \leqslant& \vv H-G^k\vvh + \leftv\frac 1 {k+g} \sum\limits_{i=1}^k\int_{S^1}G^k_t(x^k_i)dt-\int_{S^1}\int_{\Sigma}G^k_t\omega dt \rightv\\
    &+ \frac 1 {k+g}\sum\limits_{i=1}^g \int_{S^1}\max\limits_{\alpha_i} |G^k_t| dt+ \vv G^k-H\vvh\\
    \leqslant& 2\diam(\underline L^k) \sup\limits_{S^1\times\Sigma} \vv dH_t \vv + \leftv\frac 1 {k+g} \sum\limits_{i=1}^k\int_{S^1}G^k_t(x^k_i)dt-\int_{S^1}\int_{\Sigma}G^k_t\omega dt \rightv + \operatorname O\left(\frac 1 k\right)
\end{align*}

Let $A:=\Area(D^k_i)$. We have $\frac 1 {k+1}\leqslant A \leqslant \frac 1 k$, so $\diam(\underline L^k)\geqslant \frac C {\sqrt k} $ for some positive constant $C$ and therefore $\frac 1 k =\operatorname O_{k\to\infty}\left(\diam(\underline L^k) \right)$. Moreover,

\begin{align*}
    &\leftv \frac 1 {k+g} \sum\limits_{i=1}^k\int_{S^1}G^k_t(x^k_i)dt-\int_{S^1}\int_{\Sigma}G^k_t\omega dt\rightv \\
    =& \leftv\frac{1} {(k+g)A} \sum\limits_{i=1}^k\int_{S^1}\int_{D^k_i}G^k_t\omega dt -\int_{S^1}\int_{\Sigma}G^k_t\omega dt\rightv\\
    =& \leftv\left(\frac{1} {(k+g)A} -1\right) \int_{S^1}\int_{\bigcup D^k_i}G^k_t\omega dt - \int_{S^1}\int_{\Sigma\setminus(\bigcup D^k_i)}G^k_t \omega dt\rightv\\
    \leqslant& \leftv\frac 1 {(k+g)A}-1\rightv\vv G^k\vv_\infty + \vv G^k\vv_\infty \int_{\Sigma\setminus(\bigcup D^k_i)}\omega\\
    \leqslant& \left(\leftv\frac 1 {(k+g)A}-1\rightv + 1-kA\right)\left(\vv H\vv_\infty + \vv H-G^k\vv_\infty\right)\\
    =& \operatorname O\left(\frac 1 k\right)
\end{align*}

\end{proof}

Since the invariants $c_{\underline L^k}$ satisfy all the properties listed in Proposition \ref{prop:invariants}, the same proof as in \cite{CGHMSS1} show that:

\begin{prop}
    $f_{\underline L^k} := c_{\underline L^k} + \Cal : \Ham(\Sigma)\rightarrow \R$ is uniformly continuous, and therefore extends continuously to $\bHam(\Sigma)$.
\end{prop}

Now, we define the following relation on the space of real valued sequences $\mathbb R^{\mathbb N}$ : we say that $x \sim y$ if $\lim x-y = 0$. This is an equivalence relation, and the quotient $\mathbb R^{\mathbb N}/\sim$ is a real vector space.
Then, we can define a map

\begin{align*} f: &\overline{\Ham}(\Sigma) \to \mathbb R^{\mathbb N}/\sim\\
\varphi &\mapsto (f_{\underline L^1}(\varphi), f_{\underline L^2}(\varphi),...)
\end{align*}

We claim that $f$ is a group homomorphism, that is, for every $\varphi$, $\psi$ in $\overline{\Ham}(\Sigma)$:

\[f_{\underline L^k}(\varphi \psi) - f_{\underline L^k}(\varphi)-f_{\underline L^k}(\psi) \to 0\]

 Since the spectral invariants satisfy the triangle inequality, and since $\Cal$ is a group homomorphism, we have the following inequality for every $\varphi$, $\psi$ in $\Ham(\Sigma)$ and $k$ in $\mathbb N$ :
 
\[f_{\underline L^k}(\varphi \psi) - f_{\underline L^k}(\varphi)-f_{\underline L^k}(\psi)\leqslant 0 \]

This inequality still holds for the extension of $f_{\underline L^k}$ to $\overline{\Ham}(\Sigma)$.
We also have, using the triangle inequality :

\begin{align*}f_{\underline L^k}(\varphi \psi) - f_{\underline L^k}(\varphi)-f_{\underline L^k}(\psi)&\geqslant f_{\underline L^k}(\varphi \psi) - f_{\underline L^k}(\varphi)-f_{\underline L^k}(\varphi\psi)-f_{\underline L^k}(\varphi^{-1}) \\
&\geqslant -(f_{\underline L^k}(\varphi)+f_{\underline L^k}(\varphi^{-1}))
\end{align*}

Hence, the following property is enough to show that $f$ is a group homomorphism:

\begin{prop}\label{p:gammabound}
    For all $\varphi\in\bHam(\Sigma)$, $\gamma_{\underline L^k}(\varphi):=f_{\underline L^k}(\varphi)+f_{\underline L^k}(\varphi^{-1})$ goes to zero when $k$ goes to infinity.
\end{prop}

\begin{proof}
    Fix $\varphi \in \bHam(\Sigma)$. Using a standard fragmentation result (see for instance \cite{Banyaga} or \cite{S} for a more quantitative version), one can find $\varphi_1,...,\varphi_n$ supported in disks $D_1,...,D_n$ such that $\varphi = \varphi_1\circ ...\circ \varphi_n$.
    For each $1\leqslant i \leqslant n$, we pick a smooth $\psi_i$ sending $D_i$ to a disk that does not intersect $\underline\alpha:=\alpha_1\cup ..\cup,\alpha_g$.

Then we have, using the triangle inequality:

\begin{align*}0\leqslant \gamma_{\underline L^k}(\varphi) &\leqslant \sum \gamma_{\underline L^k}(\varphi_i)\\
&\leqslant \sum (\gamma_{\underline L^k}(\psi_i\varphi_i\psi_i^{-1})+\gamma_{\underline L^k}(\psi_i)+\gamma_{\underline L^k}(\psi_i^{-1}))\\
&\leqslant \sum (\gamma_{\underline L^k}(\psi_i\varphi_i\psi_i^{-1})+2\gamma_{\underline L^k}(\psi_i))
\end{align*}
Since for all $i$, $\psi_i\varphi_i\psi_i^{-1}$ is supported in a disk away from $\underline\alpha$, we can apply Theorem \ref{thm:localqm} and Remark \ref{rk:lambda} to get that $\gamma_{\underline L^k}(\psi_i\varphi_i\psi_i^{-1})\leqslant \frac {k+1} {k+g} \lambda\leqslant \frac {k+1}{k(k+g)}$, and hence this term goes to zero.

As for the other terms, since the $\psi_i$ are smooth, the Calabi property (Proposition \ref{prop:Calabi}) implies that $\gamma_{\underline L^k}(\psi_i)$ goes to $\Cal(\psi_i)+\Cal(\psi_i^{-1}) = 0$.
Thus, $\gamma_k(\varphi)$ goes to zero for any $\varphi$, and hence $f$ is a group homomorphism.
\end{proof}

For $\varphi$ smooth, we have $f(\varphi) = (\Cal(\varphi), \Cal(\varphi),...)$ since $f_{\underline L^k}(\varphi)$ converges to $\Cal(\varphi)$. Let $\Delta$ denote the vector $(1,1,1,....)$ in $\mathbb R^{\mathbb N}/\sim$. Using Zorn's lemma, we complete this vector into a base $(a_1=\Delta, a_2,a_3,...)$ of $\mathbb R^{\mathbb N}/\sim$. Now let $s$ be the following map :
\begin{align*} \mathbb R^{\mathbb N}/\sim &\to \mathbb R\\
\sum \lambda_i a_i &\mapsto \lambda_1
\end{align*}
Then $s\circ f$ is a group homomorphism from $\overline{\Ham}(\Sigma)$ to $\mathbb R$ that extends $\Cal$.
This completes the proof of Theorem \ref{thm:extensiongeneral}.

\subsection{Proof of Theorem \ref{thm:notsimplegeneral}}
We now give a proof of Theorem \ref{thm:notsimplegeneral}. Once again, it is inspired of the proof of (\ref{thm:notsimple}) found in \cite{CGHMSS2}.

We start by fixing an equidistributed sequence of links $\underline L^k = L^k_1\cup ...\cup L^k_k \cup\alpha_1\cup ...\cup\alpha_g$, such that $\diam(\underline L^k)=\operatorname{O}\left(\frac 1 {\sqrt k}\right)$.
We define
\[N((\underline L^k)_{k\in\N^*}):=\lbrace\varphi\in \Hameo(\Sigma,\omega)|\sqrt k(f_{\underline L^k} \varphi)-\Cal(\varphi)) \text{ is bounded} \rbrace\]
We claim:

\begin{prop}
\label{prop:normalsubgrp}
$N((\underline L^k)_{k\in\N^*})$ is a normal sub-group of $\Hameo(\Sigma,\omega)$
\end{prop}

\begin{proof}
    Let $\varphi,\psi\in N((\underline L^k)_{k\in\N^*})$, and $\theta\in\Hameo(\Sigma)$. Then, by the triangle inequality and the fact that $\Cal$ is a homomorphism,
    \[-\sqrt k \gamma_{\underline L^k}(\varphi)\leqslant \sqrt k(f_{\underline L^k}(\varphi\psi)-\Cal(\varphi\psi))- \sqrt k(f_{\underline L^k}(\varphi)-\Cal(\varphi))- \sqrt k(f_{\underline L^k}(\psi)-\Cal(\psi))\leqslant 0\]

    Moreover, we have that    
    \[\sqrt k (f_{\underline L^k}(\varphi^{-1})-\Cal(\varphi^{-1}))=\sqrt k\gamma_{\underline L^k}(\varphi) - \sqrt k(f_{\underline L^k}(\varphi)-\Cal(\varphi))\]
    
    and the triangle inequality also implies that    
    \[-\sqrt k \gamma_{\underline L^k}(\theta)\leqslant\sqrt k(f_{\underline L^k}(\varphi)-f_{\underline L^k}(\theta\varphi\theta^{-1}))\leqslant \sqrt k \gamma_{\underline L^k}(\theta)\]

    Therefore, the following lemma proves that $\varphi\psi$, $\varphi^{-1}$ and $\theta\varphi\theta^{-1}$ are in $N((\underline L^k)_{k\in\N^*})$, which concludes the proof of the proposition.
    
    \begin{lemma}
        \label{lem:gamma}
        For all $\varphi\in\bHam(\Sigma)$, $( \sqrt k \gamma_{\underline L^k}(\varphi))$ is bounded.
    \end{lemma}

    \begin{proof}



        Writing $\varphi = \varphi_1...\varphi_N$ where each $\varphi_i$ is supported in a disk $D_i$, and choosing some $\psi_i$ displacing $D_i$ away from $\underline \alpha$, we get as in the proof of Proposition \ref{p:gammabound} that
\begin{align*} \sqrt k\gamma_{\underline L^k}(\varphi)&\leqslant \sqrt k\sum\limits_{i=1}^N(\gamma_{\underline L^k}(\psi_i\varphi_i\psi_i^{_1})+2\gamma_{\underline L^k}(\psi_i))\\
&\leqslant \sqrt k\sum\limits_{i=1}^N\left(\frac {k+1} {k(k+g)} + \operatorname O\left(\frac 1 { \sqrt k}\right)\right)\\
&\leqslant \frac{k+1}{\sqrt k(k+g)}N+\operatorname O(1)
\end{align*}
where we used the Calabi property (Proposition \ref{prop:Calabi}), with $\diam(\underline L^k)=\operatorname{O}\left(\frac 1 {\sqrt k}\right)$.
This shows that $\sqrt k\gamma_{\underline L^k}(\varphi)$ is bounded for every $\varphi$.
    \end{proof}

\end{proof}

\begin{remark}We see in this proof that the terms $\gamma_{\underline L^k}(\psi_i)$ are of higher order than the other ones. If we manage to show that for smooth elements, $k\gamma_{\underline L^k}$ is bounded, then it would be the case for all $\varphi\in\bHam(\Sigma)$, and we could define an even smaller normal subgroup by considering $k\gamma_{\underline L^k}$ instead of $\sqrt k\gamma_{\underline L^k}$.
\end{remark}

\medskip
It remains to show that for a certain choice of $(\underline L^k)$, this subgroup is proper. The proof is similar to the one in \cite{CGHMSS2} in  the case of the disk. There are three steps:

\begin{itemize}
\item We show that $N((\underline L^k)_{k\in\N^*})\cap \Ker(\Cal)$ contains all the smooth elements;
\item We construct a hameomorphism $T$, and choose an equidistributed sequence $(\underline L^k)$, such that $T$ does not belong to $N((\underline L^k)_{k\in\N^*})$.
\item From this $T$ we can construct another hameomorphism with the same property that lies in $\Ker(\Cal)$.
\end{itemize}

\begin{lemma}
\label{lem:smooth}
$N((\underline L^k)_{k\in\N^*})\cap \Ker(\Cal)$ contains all the smooth elements.
\end{lemma}

\begin{proof}


It is a corollary of the Calabi property (Proposition \ref{prop:Calabi}).
\end{proof}

Now we construct the hameomorphism and the sequence of links.
Fix $g$ disjoint and homologically independant circles $\alpha_1,..., \alpha_g$. Let $D$ be a disk of area $1/2$ in $\Sigma$ away from $\underline \alpha$, and pick a point $z_0$ in its interior. We fix a symplectomorphism $\Phi: (D\setminus\{z_0\},\omega)\xrightarrow{\sim} \left(S^1\times \left(0,\frac 1 {\sqrt{2\pi}}\right],rdr\wedge d\theta\right)$.

We define an autonomous Hamiltonian $H$ on $\Sigma\setminus\{z_0\}$ as follow:

\begin{itemize}
    \item $H$ is supported inside $\Phi^{-1}\left(S^1\times \left(0,\frac 1 {2\sqrt{\pi}}\right]\right)\subset D\setminus\{z_0\}$;
    \item $H(\theta,r)=h(\pi r^2)$ is radial;
    \item $h: \left(0,\frac 1 4\right]\to [0,+\infty)$ is decreasing, $h(r)\leqslant r^{-a}$ with equality on $\left(0,\frac 1 8\right]$ for some $\frac 1 2 + \frac 1 {2\sqrt{2}}< a< 1$.
\end{itemize}

Then, $\varphi^1_H$ defines a Hamiltonian diffeomorphism on $\Sigma\setminus\{z_0\}$ which acts as a rotation around the origin inside $D\setminus\{z_0\}$. Therefore, it extends continuously to a homeomorphism $T$ that fixes $z_0$. We claim the following:

\begin{prop}
    $T\in\Hameo(\Sigma)$
\end{prop}

\begin{proof}
    We have to find a sequence of Hamiltonians $(K_n)$, supported in a compact subset of the interior of $\Sigma$, such that:
    
    \begin{itemize}
        \item $\phi^1_{K_n}\xrightarrow{C^0}T$;
        \item $(\phi^t_{K_n})$ is Cauchy for the $C^0$ distance, uniformly in $t\in[0,1]$;
        \item the sequence $(K_n)$ is Cauchy for the Hofer norm.
    \end{itemize}

    Let $D_n:=\{z_0\}\cup\Phi^{-1}\left(S^1\times\left(0,\frac 1 {\sqrt {\pi 2^{n/a}}}\right)\right)\subset\Sigma$. It has area $\frac 1 {2^{n/a}}$.
    
    We start with a sequence of smooth Hamiltonians $(H_n)$ such that:
    \begin{itemize}
        \item $H_n$ coincides with $H$ outside of $D_n$;
        \item $H_n \approx 2^{n}$ in $D_n$
        \item $\vv H_{n+1}-H_n\vvh\leqslant 2^n$
    \end{itemize}
    To construct such a sequence, we flatten $H$ inside $D_n$.

    Since $H_n$ coincides with $H$ outside of $D_n$, we have that $\phi^1_{H_n}\circ T^{-1}=\Id$ outside of $D_n$, and therefore $\phi^1_{H_n}\xrightarrow{C^0}T$. 

    We will now construct a sequence $(K_n)$ such that $\phi^1_{K_n}=\phi^1_{H_n}$, $(K_n)$ is Cauchy for the Hofer norm, and $(\phi^t_{K_n})$ is Cauchy for the $C^0$ distance uniformly in $t$.

        We will use a lemma from \cite[Lemma 4.5]{CGHMSS2}:

    \begin{lemma}
        Let $\Delta$ be a Euclidean $2$-disk equipped with an area form $\omega$ of total area A.
    Suppose $D\subset\Delta$ is diffeomorphic to $D^2$ and that $\Area(D)<\frac A N$ some integer $N > 0$.
    Let F be a smooth Hamiltonian supported in the interior of $D$. Then, we have:
    \[d_{H}(\phi^1_F,\Id)\leqslant \frac {\vv F\vvh} N + 2A\]
    where $d_H$ denotes the Hofer distance on $\Ham_c(\Delta,\omega)$.
    \end{lemma}

    Let $b$ be a real number such that $1<b<\frac 1 a$. Let $N=2^{\left\lfloor bn\right\rfloor}$, and $A_n=(N+1)2^{-\frac n a}$. If $n\geqslant n_0$ where $n_0$ is large enough, $A<\frac 1 2$, and we can define $\Delta_n:=\{z_0\}\cup\Phi^{-1}\left(S^1\times\left(0,\frac 1 {\sqrt {\pi A_n}}\right)\right)$. It is a disk of area $A_n$. $H_{n+1}-H_n$ is supported inside $D_n$, which has area $2^{-\frac n a}<\frac {A_n} N$, so we can apply the lemma and get that:

    \[d_{H}(\phi^1_{H_{n+1}-H_n},\Id)\leqslant \frac {\vv H_{n+1}-H_n\vvh} N + 2A_n\]

    Therefore there exists $G_n$ supported in $\Delta_n$ such that $\phi^1_{G_n} = \phi^1_{H_{n+1}-H_n} = \phi^{-1}_{H_n}\circ\phi^1_{H_{n+1}}$ and $\vv G_n\vvh \leqslant 2^{n-\left\lceil (1-b)n\right\rceil}+(2^{\left\lfloor bn\right\rfloor}+1)2^{-\frac a n}$.

    By definition of $b$, the series $\sum\limits_{n=n_0}^{\infty}\vv G_n\vvh$ is summable.
    Since $G_n$ is supported inside $\Delta_n$, $d_{C^0}(\phi^t_{G_n},\Id)\leqslant \diam \Delta_n = \operatorname O(2^{(b-\frac 1 a)n})$, so $\sum\limits_{n=n_0}^{\infty}d_{C^0}(\phi^t_{G_n},\Id)$ is also summable, uniformly in $t$.
    
    Then, we define $(K_n)$ recursively by:
    \begin{itemize}
        \item $K_n=H_n$ for $n\leqslant n_0$;
        \item $K_{n+1}=K_n\#G_n$ for $n\geqslant n_0$
    \end{itemize}

    We get that $\phi^1_{K_n}=\phi^1_{H_n}$ for $n\leqslant n_0$, and for $n> n_0$:
    \begin{align*}
        \phi^1_{K_n}&=\phi^1_{H_{n_0}}\phi^1_{G_{n_0}}...\phi^1_{G_{n-1}}\\
        &=\phi^1_{H_{n_0}}\phi^{-1}_{H_{n_0}}\phi^1_{H_{n_0+1}}...\phi^{-1}_{H_{n-1}}\phi^1_{H_{n}}\\
        &= \phi^1_{H_{n}}
    \end{align*}

    Moreover, the summability of $\sum\limits_{n=n_0}^{\infty}\vv G_n\vvh$ and $\sum\limits_{n=n_0}^{\infty}d_{C^0}(\phi^t_{G_n},\Id)$ implies that $(K_n)$ is Cauchy for the Hofer norm, and $(\phi^t_{K_n})$ is Cauchy for the $C^0$ distance uniformly in $t$.

    This concludes the proof that $T\in \Hameo(\Sigma)$.
    
\end{proof}

We now construct an equidistributed sequence of admissible links as follow:

Fix an integer $k\geqslant 1$.
For $0\leqslant i\leqslant \lfloor\frac {\sqrt {k}} 2\rfloor$, denote by $A_i$ the annulus $S^1\times \left(\frac i {\sqrt{\pi k}}, \frac {i+1} {\sqrt{\pi k}}\right)\subset D\setminus \{z_0\}$.

Let $L^k_1$ be the circle $S^1\times \left\{\sqrt{\frac {1}{\pi (k+1)}}\right\}$. It bounds a disk of area $\frac 1 {k+1}$.
For $i\geqslant 1$, each annulus $A_i$ has area $\frac{1} {k}((i+1)^2-i^2)=\frac{2i+1} k$, hence we can fit inside $A_i$ $2i+1$ disjoint circles $L^k_{i^2+1},..., L^k_{(i+1)^2}$ that bound disjoint disks of area $\frac{1}{k+1}$ and of diameter bounded by $\frac C {\sqrt k}$ where $C$ is a constant that does not depend on $k$.

The union of all the annuli cover a disk of area $\frac{\left(\lfloor\frac {\sqrt {k}} 2\rfloor+1\right)^2}{k}$. The remaining area in $\Sigma$ is $\frac{k-\left(\lfloor\frac {\sqrt {k}} 2\rfloor+1\right)^2}{k}$, which is enough to fit $k-\left(\lfloor\frac {\sqrt {k}} 2\rfloor+1\right)^2$ disjoint circles $L^k_{\left(\lfloor\frac {\sqrt {k}} 2\rfloor+1\right)^2+1}, ..., L^k_k$ that bound disjoint disks of area $\frac{1}{k+1}$ and of diameter bounded by $\frac {C'} {\sqrt k}$ where $C'$ is a constant that does not depend on $k$.

\begin{figure}
	\centering
	\includegraphics[scale=0.6]{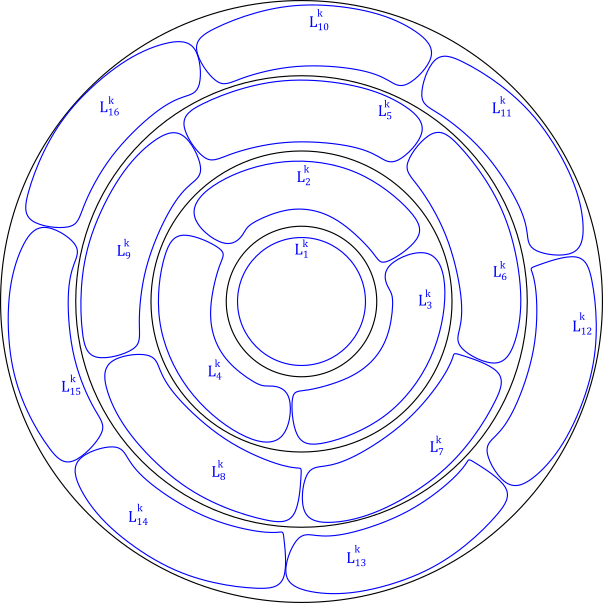}
	\caption{The sequence $\underline L^k$}
	\label{fig:Figure 3}
\end{figure}

Let $\underline L^k:= L^k_1\cup...\cup L^k_k\cup\alpha_1\cup...\alpha_g$. Then $(\underline L^k)$ is an equidistributed sequence of monotone Lagrangian links, and moreover:

\begin{prop}
    For this choice of equidistributed sequence, $T\notin N((\underline L^k)_{k\in\N^*})$.
\end{prop}

\begin{proof}
    We want to show that $\sqrt k(f_{\underline L^k}(T)-\Cal(T))$ is unbounded.

    First, we observe that:
    
\begin{align*}
    \Cal(T)&=\lim\limits_{n\to\infty}\Cal(\phi^1_{K_n})
    = \lim\limits_{n\to\infty}\Cal(\phi^1_{H_n})
    = \lim\limits_{n\to\infty} \int_\Sigma H_n \omega
    = \int_\Sigma H\omega
\end{align*}

and

\begin{align*}
    f_{\underline L^k}(T)=\lim\limits_{n\to\infty}f_{\underline L^k}(\phi^1_{K_n})
    = \lim\limits_{n\to\infty}f_{\underline L^k}(\phi^1_{H_n})
    = \lim\limits_{n\to\infty} c_{\underline L^k}(H_n)
\end{align*}

Since $H_n$ coincides with $H$ outside of $D_n$, for $n$ sufficiently large, $H_n$ coincides with $H$ on $\underline L^k$ and therefore $c_{\underline L^k}(H_n)=c_k(H)$. Thus, $f_{\underline L^k}(T)=c_{\underline L^k}(H)$.

    We start by estimating $c_{\underline L^k}(H)$. Since $H$ is supported inside $S^1\times \left(0,\frac 1 {2\sqrt{\pi}}\right]$, and by Lagrangian control:

    \begin{align*}
    c_{\underline L^k}(H) &\leqslant \frac 1 {k+g} \sum\limits_{i=1}^{\left(\lfloor\frac {\sqrt {k}} 2\rfloor+1\right)^2}\max\limits_{L^k_i}H\\
    &\leqslant \frac 1 {k} \left(\max\limits_{L^k_1}H+\sum\limits_{i=1}^{\lfloor\frac{\sqrt k}2\rfloor}(2i+1)\max\limits_{A_i}H\right)\\
    &\leqslant \frac 1 {k} \left(h\left(\frac 1 {k+1}\right)+\sum\limits_{i=1}^{\lfloor\frac{\sqrt k}2\rfloor}(2i+1)h\left(\frac {i^2} k\right)\right)\\
    &\leqslant \frac 1 {k}\left(h\left(\frac 1 {k+1}\right)+3h\left(\frac 1 {k}\right)+2\sum\limits_{i=2}^{\lfloor\frac{\sqrt k}2\rfloor}h\left(\frac {i^2} k\right)+\sum\limits_{i=2}^{\lfloor\frac{\sqrt k}2\rfloor}(2i-1)h\left(\frac {i^2} k\right)\right)\\
    \end{align*}
Using that $h$ is decreasing, and comparing the sums with integrals, we get:
    \begin{align*}
    c_{\underline L^k}(H) &\leqslant \frac{(k+1)^a} k +3k^{a-1}+\frac 2 {\sqrt k}\int_{\frac 1 {\sqrt k}}^{\frac 1 2}h(r^2)dr+\sum\limits_{i=1}^{\lfloor\frac{\sqrt k}2\rfloor-1}\frac{2i+1} k \min\limits_{A_i}H\\
    &\leqslant \left(\frac{k+1} k\right)^a k^{a-1} +3k^{a-1}+\frac 2 {\sqrt k}\left[\frac 1 {1-2a}r^{1-2a}\right]_{\frac 1 {\sqrt k}}^{\frac 1 2} + \int_{\Sigma\setminus A_0}H\omega\\
    &\leqslant \left(1+\frac{a} k\right) k^{a-1} +3k^{a-1}+\frac 2 {2a-1}k^{a-1} + \int_{\Sigma\setminus A_0}H\omega
    \end{align*}

    Therefore, for $k\geqslant 8$,

    \begin{align*}
    \sqrt k (f_{\underline L^k}(T)-\Cal(T))&=\sqrt k (c_{\underline L^k}(H)-\sqrt k\int_\Sigma H\omega)\\
    &\leqslant \left(4+\frac a k +\frac 2 {2a-1}\right)k^{a-\frac 1 2} -\sqrt k\int_{A_0}H\omega\\
    &\leqslant \left(4+\frac a k+\frac 2 {2a-1}\right)k^{a-\frac 1 2} -\sqrt k\int_0^{\frac 1 k}h(r)dr\\
    &\leqslant \left(4+\frac a k+\frac 2 {2a-1}\right)k^{a-\frac 1 2} - \sqrt k\frac 1 {1-a}k^{a-1}\\
    &\leqslant \left(4+\frac a k+\frac 2 {2a-1}-\frac 1 {1-a}\right)k^{a-\frac 1 2}\\
    &\leqslant \left(\frac{-8a^2+8a-1}{(2a-1)(1-a)}+\frac a k\right)k^{a-\frac 1 2}\\
    \end{align*}

    Since $\frac 1 2 + \frac 1 {2\sqrt{2}}<a<1$, $\frac{-8a^2+8a-1}{(2a-1)(1-a)}+\frac a k<0$ for $k$ large enough, and $\sqrt k (f_k(T)-\Cal(T))$ goes to $-\infty$.
    
\end{proof}

In a similar fashion, we also compute a lower bound for $c_{\underline L^k}(H)-\int_\Sigma H\omega$ (when $k$ is large enough) that we will need in the following section:

    \begin{align*}
    c_{\underline L^k}(H) -\int_\Sigma H\omega
    \geqslant& \frac 1 {k+g} \sum\limits_{i=1}^{\left(\lfloor\frac {\sqrt {k}} 2\rfloor+1\right)^2}\min\limits_{L^k_i}H -\int_\Sigma H\omega\\
    \geqslant& \frac 1 {k+g} \left(\min\limits_{L^k_1}H+\sum\limits_{i=1}^{\lfloor\frac{\sqrt k}2\rfloor}(2i+1)\min\limits_{A_i}H\right)-\int_\Sigma H\omega\\
    \geqslant& \frac 1 {k+g} \left(h\left(\frac 1 {k+1}\right)+\sum\limits_{i=1}^{\lfloor\frac{\sqrt k}2\rfloor}(2i+1)h\left(\frac {(i+1)^2} k\right)\right)-\int_\Sigma H\omega\\
    \geqslant& \frac 1 {k+g}\left(h\left(\frac 1 {k+1}\right)-3h\left(\frac 1 {k}\right)-2\sum\limits_{i=2}^{\lfloor\frac{\sqrt k}2\rfloor}h\left(\frac {i^2} k\right)+\sum\limits_{i=1}^{\lfloor\frac{\sqrt k}2\rfloor}(2i+1)h\left(\frac {i^2} k\right)\right)-\int_\Sigma H\omega\\
    & \frac k {k+g}\left(\frac{(k+1)^a} k -3k^{a-1}-\frac 2 {\sqrt k}\int_{\frac 1 {\sqrt k}}^{\frac 1 2}h(r^2)dr+\sum\limits_{i=1}^{\lfloor\frac{\sqrt k}2\rfloor-1}\frac{2i+1} k \max\limits_{A_i}H-\int_\Sigma H\omega\right)\\
    &-\frac g {k+g}\int_\Sigma H\omega\\
            \end{align*}
            Once again, we used that $h$ is decreasing and compared the first sum with an integral. Doing the same for the second sum, we obtain:
\begin{align*}
    c_{\underline L^k}(H) -\int_\Sigma H\omega\geqslant& \frac 1 {1+\frac g k}\left(\left(\frac{k+1} k\right)^a k^{a-1} -3k^{a-1}-\frac 2 {\sqrt k}\left[\frac 1 {1-2a}r^{1-2a}\right]_{\frac 1 {\sqrt k}}^{\frac 1 2} + \int_{\Sigma\setminus A_0}H\omega-\int_\Sigma H\omega\right)\\
    &-\frac g {k}\int_0^{\frac 1 4} h(r)dr\\
    \geqslant& \left(1-\frac g k\right)\left(\left(1+\frac{1} k\right)^a k^{a-1} -3k^{a-1}-\frac 2 {2a-1}k^{a-1} - \int_{A_0}H\omega\right)-\frac g {k}\int_0^{\frac 1 4}r^{-\alpha}dr\\
    \geqslant& \left(1-\frac g k\right)\left(\left(1 - 3 -\frac 2 {2a-1}\right)k^{a-1} -\int_0^{\frac 1 k}h(r)dr\right)-\frac g {k}\left[\frac {r^{1-\alpha}}{1-\alpha}\right]^{\frac 1 4}_0\\
    \geqslant& \left(1-\frac g k\right)\left(\left(-2-\frac 2 {2a-1}\right)k^{a-1} -\frac 1 {1-a}k^{a-1}\right)-\frac {4^{\alpha-1}g} {(1-\alpha)k}\\
    \geqslant& \left(1-\frac g k\right)\left(-2-\frac 2 {2a-1}-\frac 1 {1-a}\right)k^{a-1}-\frac {4^{\alpha-1}g} {(1-\alpha)k}\\
    \end{align*}

Now, it remains to construct a hameomorphism in $\Ker(\Cal)$. Choose a smooth Hamiltonian diffeomorphism $\theta$ such that $\Cal(\theta)=\Cal(T)$. Then, $T':=T\theta^{-1}\in\Hameo(\Sigma)\cap\Ker(\Cal)$.
Since $\theta$ is smooth, by Lemma \ref{lem:smooth}, $\theta\in N(\Sigma)$, and therefore $T'\notin N((\underline L^k)_{k\in\N^*})$.

Hence, $N((\underline L^k)_{k\in\N^*})\cap\Ker(\Cal)$ is a proper normal subgroup of $\Hameo(\Sigma)\cap\Ker(\Cal)$, which concludes the proof of Theorem \ref{thm:notsimplegeneral}.

We give more precision on the subgroups we defined:

Note that since the Hamiltonian $H$ we constructed in this section is radial, we have $H^{\# n}=nH$, and therefore
\[\mu_{\underline L^k}(H) = \lim\limits_{n\to\infty}\frac {c_{\underline L^k}(H^{\# n})} n = c_{\underline L^k}(H)\]

Then, by Theorem \ref{thm:mu}, $c_{\underline L^k}(H)$ only depends on $k$, $\underline \alpha$ and the monotonicity constant of the link. Therefore, for any other equidistributed sequence of links $(\underline L'^k)$ having the same non-contractible components $\underline \alpha$, and the same monotonicity constant, we have that $T \notin N_{(\underline L'^k)}$.

Since we fixed an arbitrary $\underline \alpha$ at the start of the proof, and since we could modify the definition of $(\underline L^k)$ to change its monotonicity constant while keeping similar inequalities for $c_{\underline L^k}(H)$, we get the following proposition:

\begin{prop}Let $(\underline L'^k)$ be an equidistributed sequence of admissible links, satisfying $\diam(\underline L'^k)=\operatorname{O}\left(\frac 1 {\sqrt k}\right)$. Then, $N((\underline L'^k)_{k\in\N^*})$ is a proper normal subgroup of $\Hameo(\Sigma)\cap\Ker(\Cal)$.

Moreover, taking the intersection of those subgroups over all such sequences of links, we get an even smaller proper normal subgroup $N$, which contains all smooth elements.
\end{prop}

\begin{remark}
    We need the assumption on the diameter to ensure that $N((\underline L'^k)_{k\in\N^*})$ is a normal subgroup (Proposition \ref{prop:normalsubgrp}).
\end{remark}

\subsection{Proof of Theorem \ref{thm:hameogeneral}}

This time we consider a connected, closed, oriented surface $(\Sigma,\omega)$.

Fix an equidistributed sequence of admissible links $(\underline L^k)$ with $\diam(\underline L^k)=\operatorname{O}\left(\frac 1 {\sqrt k}\right)$.

Fix $a$ such that $\frac 1 2 + \frac 1 {2\sqrt{2}}<a<1$.

Then, for $k_0$ large enough, $\frac 1 {1-a}-4-\frac a {2^{2k_0}}-\frac 2 {2a-1}>0$. Fix such a $k_0$, then for $N$ large enough, $-2-\frac 2 {2a-1}-\frac 1 {1-a}+2^{N(1-a)}\left(\frac 1 {1-a}-4-\frac a {2^{2k_0}}-\frac 2 {2a-1}\right)>0$.

Fix such an integer, and define $g_k:=c_{\underline L^{2^{2k}}}-c_{\underline L^{2^{2k-N}}}$.

Let $N((\underline L^k)_{k\in\N^*}):=\{\varphi \in \Hameo(\Sigma), (2^kg_k(\varphi)) \text{ is bounded}\}$

\begin{prop}
    $N((\underline L^k)_{k\in\N^*})$ is a normal subgroup of $\Hameo(\Sigma)$.
\end{prop}

\begin{proof}
    Let $\varphi,\psi\in N((\underline L^k)_{k\in\N^*})$, and $\theta\in\Hameo(\Sigma)$.
    Using the triangle inequality, we have:
    \[\gamma_{\underline L^{2^{2k}}}(\varphi)\leqslant g_k(\varphi\psi)-g_k(\varphi)-g_k(\psi)\leqslant \gamma_{\underline L^{2^{2k-N}}}(\varphi)\]
    \[g_k(\varphi^{-1})=\gamma_{\underline L^{2^{2k}}}(\varphi)-\gamma_{\underline L^{2^{2k-N}}}(\varphi)-g_k(\varphi)\]
    and
    \[-(\gamma_{\underline L^{2^{2k}}}(\theta)+\gamma_{\underline L^{2^{2k-N}}}(\theta))\leqslant g_k(\theta\varphi\theta^{-1})-g_k(\varphi)\leqslant \gamma_{\underline L^{2^{2k}}}(\theta)+\gamma_{\underline L^{2^{2k-N}}}(\theta)\]
    Therefore, Lemma \ref{lem:gamma} implies that $\varphi\psi$, $\varphi^{-1}$ and $\theta\varphi\theta^{-1}$ are in $N(\Sigma)$.
\end{proof}

We claim that for a certain choice of link, $N((\underline L^k)_{k\in\N^*})$ is a proper subgroup of $\Hameo(\Sigma)$.

In fact, once again 
Proposition \ref{prop:Calabi} shows that it contains all the smooth elements.

We define a hameomorphism $T$ and a sequence of links $(\underline L^k)$ as in the previous section, with the parameter $\frac 1 2 + \frac 1 {2\sqrt{2}}< a < 1$ we fixed earlier.

We claim that $T \notin N((\underline L^k)_{k\in\N^*})$.

Let $k$ be a sufficiently large integer. Then, by the estimates of the previous section:

\begin{align*}
    2^kg_k(T)&=2^k\left(c_{\underline L^{2^{2k}}}(H)-c_{\underline L^{2^{2k-N}}}(H)\right)\\
    &=2^k\left(\left(c_{\underline L^{2^{2k}}}(H)-\int_\Sigma H\omega\right)+\left(\int_\Sigma H\omega-c_{\underline L^{2^{2k-N}}}(H)\right)\right)\\
    &\geqslant 2^k\left(\left(1-\frac g {2^{2k}}\right)\left(-2-\frac 2 {2a-1}-\frac 1 {1-a}\right)(2^{2k})^{a-1}-\frac {4^{\alpha-1}g} {(1-\alpha)2^{2k}}\right)\\
    &+2^k\left(\frac 1 {1-a}-4-\frac a {2^{2k-N}}-\frac 2 {2a-1}\right)(2^{2k-N})^{a-1}\\
    &\geqslant 2^{k(2a-1}\left(-2-\frac 2 {2a-1}-\frac 1 {1-a}+2^{N(1-a)}\left(\frac 1 {1-a}-4-\frac a {2^{2k-N}}-\frac 2 {2a-1}\right)\right)\\
    &-2^{k(2a-3)}g\left(-2-\frac 2 {2a-1}-\frac 1 {1-a}\right)-2^{-k}\frac {4^{\alpha-1}g} {1-\alpha}
\end{align*}

By definition of $N$, for $k\geqslant k_0+\frac N 2$, $-2-\frac 2 {2a-1}-\frac 1 {1-a}+2^{N(1-a)}\left(\frac 1 {1-a}-4-\frac a {2^{2k-N}}-\frac 2 {2a-1}\right)>0$, and therefore $2^kg_k(T)$ goes to infinity, which concludes the proof of Theorem \ref{thm:hameogeneral}.

Using the same argument as in the previous section, we also get:

\begin{prop}Let $(\underline L'^k)$ be an equidistributed sequence of admissible links satisfying $\diam(\underline L'^k)=\operatorname{O}\left(\frac 1 {\sqrt k}\right)$. Then, $N((\underline L'^k)_{k\in\N^*})$ is a proper normal subgroup of $\Hameo(\Sigma)$.

Moreover, taking the intersection of those subgroups over all such sequences of links, we get an even smaller proper normal subgroup $N$, which contains all smooth elements.
\end{prop}

\section{The Künneth formula for connected sums}
\label{sec:stabilization}

This section is devoted to the proof of Theorem \ref{thm:stabilization}.

\subsection{Heeagaard Floer Homology}
\label{sec:Heegaard}

Let us start by discussing how we define Heegaard Floer Homology. Indeed, we decided to use the original construction, which computes Lagrangian Floer Homology in the symmetric product. Alternatively, one could work in a cylindrical setting and define Heegaard Floer Homology by counting pseudo-holomorphic curves in the 4-manifold $\Sigma\times [0,1]\times \R$.
We believe that this cylindrical reformulation (formulated by Lipshitz in \cite{Lipshitz}) could be used to prove the statements of this paper since similar results are proved in \cite{Lipshitz} and \cite{OS2} in the cylindrical setting.
Indeed, a cylindrical reformulation of the Lagrangian link spectral invariants is considered in \cite{Chen21}, \cite{Chen22} so it is likely that the cylindrical approach together with our arguments in the previous sections can be combined to obtain Theorem \ref{thm:extensiongeneral}, \ref{thm:notsimplegeneral} and \ref{thm:hameogeneral}.
However, since our main results are inspired by \cite{CGHMSS1}, \cite{CGHMSS2} and \cite{OS}, which are all using the symmetric product setting, we will do the same. This setting is the following.

Consider a closed symplectic surface $(\Sigma,\omega)$, with a compatible complex structure $j$.
Let $\underline L = L_1\cup...\cup L_k$ be a Lagrangian link in $\Sigma$.
Denote by $\Sym \underline L$ the image of $L_1\times ... \times L_k$ in the symmetric product $\Sym^k(\Sigma) := \Sigma^k/\mathfrak S_k$.

Denote by $\pi$ the projection $\Sigma^k\to\Sym^k(\Sigma)$. The symmetric product is naturally endowed with a singular symplectic form $\Sym(\omega):=\pi_*(\omega^{\oplus k})$, which is smooth away from the diagonal $\Delta := \pi\left(\{(x_1,...,x_k),\exists i \neq j, x_i = x_j\}\right)$. It is also endowed with a complex structure $\Sym^k(j):=\pi_* j^{\times k}$.

Since the circles composing $\underline L$ are disjoint, $\Sym \underline L$ does not intersect $\Delta$. Let $V$ be a neighborhood of $\Delta$ that does not intersect $\Sym \underline L$. Then, one can find a smooth symplectic form $\omega_V$ on $\Sym^k(\Sigma)$ which agrees with $\Sym(\omega)$ away from $V$, and compatible with $\Sym^k(j)$. Then, $\Sym \underline L$ is a Lagrangian submanifold inside $(\Sym^k(\Sigma),\omega_V)$.

Let $(\underline L,\underline K)$ be a pair of Lagrangian links with $k$ components.
Let $H:S^1\times \Sigma \to \R$ be a smooth Hamiltonian. Then one can define a Hamiltonian $\Sym^k(H)$ on $\Sym^k(\Sigma)$ by the formula
$\Sym^k(H)_t(\{x_1,...,x_k\}):=H_t(x_1)+...+H_t(x_k)$.

Let $V$ be a neighborhood of $\Delta$ that does not intersect $\Sym (\varphi_{H}^t (\underline L))$, $\Sym (\varphi_{H}^t (\underline K))$ for $t\in[-1,1]$.

\begin{definition}
    An almost complex structures $J$ over $\Sym^k(\Sigma)$ is $V$-nearly symmetric if it agrees with $\Sym^k(j)$ over $V$, and tames $\Sym(\omega)$ outside of $V$.
\end{definition}

If $\Sym \underline L$ and $\Sym \underline K$ are both monotone Lagrangian submanifolds that are Hamiltonian isotopic and $\Sym (\varphi_{H}^1 (\underline L)) \pitchfork \Sym \underline K$, then given a path of $V$-nearly symmetric almost complex structures $(J_t)_{t \in [0,1]}$, one can define the Lagrangian Floer cohomology 
\[HF^*(\Sym\underline L,\Sym\underline K,\Sym^k(H))\]
in the standard way (cf. \cite[Section 6]{CGHMSS1}). One can show that this construction will not depend on the choice of $V$ and $J_t$.

In order to clarify some notations, we will recall briefly how $HF^*(L,K,H)$ is constructed for two Hamiltonian isotopic monotone Lagrangians $L$ and $K$ inside a closed monotone symplectic manifold $(M,\omega)$, and a smooth Hamiltonian $H$.


We define the space $\mathcal P(L, K)$ of smooth paths $\gamma: [0,1]\to M$ with $\gamma(0)\in L$ and $\gamma(1)\in K$.

Fix $\eta$ in $\mathcal P(L,K)$, and let $\widetilde{P}_\eta(L,K)$ be the universal cover of the connected component of $\eta$ (with base point $\eta$).

Given a Hamiltonian $H$, we can define an action functional on $\widetilde{P}_\eta(L,K)$, by :
\[\mathcal A_H([\gamma, w]):= - \int w^*\omega + \int H\left(t,\gamma (t)\right)dt\]
Here, $w$ is a homotopy from $\gamma$ to $\eta$ in $\mathcal P(L,K)$.

By the definition of $\widetilde{P}_\eta(L,K)$, two cappings $[\gamma,w]$ and $[\gamma',w']$ are isomorphic if $\gamma=\gamma'$ and $w$ and $w'$ coincide in the set $\pi_2(\eta,\gamma)$ of homotopy classes of cappings from $\eta$ to $\gamma$ with boundary in $L$ and $K$.

\begin{definition}\label{d:equivalentCap}
    Two cappings $[\gamma,w]$ and $[\gamma',w']$ are defined to be equivalent if $\gamma=\gamma'$ and $w$ and $w'$ have the same image in $\frac{\pi_2(\eta,\gamma)}{\Ker \omega}$.
\end{definition}

Let $CF^*_{\circ}(L,K,H;\eta)$ be the $\mathbb{C}$-vector space generated by the equivalent classes of critical points of the action functional $\mathcal A_H$.
It is naturally a $\mathbb{C}[T^{\pm 1}]$-module where $T$ acts by adjoining the smallest positive area disk class in $\pi_2(\eta,\eta)$ to the capping.
The Lagrangian Floer complex is 
\begin{align*}
CF^*(L,K,H)&:= \oplus_{\eta \in \pi_0(\mathcal P(L,K))} CF^*(L,K,H;\eta)\\
CF^*(L,K,H;\eta)&:=CF^*_{\circ}(L,K,H;\eta) \otimes_{\mathbb{C}[T^{\pm 1}]} \mathbb{C}[[T]][T^{-1}].
\end{align*}
One can also think of $CF^*(L,K,H)$ as a $\mathbb{C}[[T]][T^{-1}]$-vector space generated by the critical points of the circle-valued action functional on $\mathcal P(L,K)$ descended from $\mathcal A_H$. These critical points are trajectories of $\varphi_H^t$ from $L$ to $K$,
and are in one-to-one correspondance with 
$\varphi^1_H(L)\cap K$. This complex is graded by the Maslov index, and the Novikov parameter carries a grading given by the minimal Maslov number of $L$.

To define the differential, we fix a path of $\omega$-compatible almost complex structure $J_t$ on $M$. Then, given two paths $\gamma$ and $\gamma'$, and a homotopy class $\beta$ of Maslov index 1, we define the space $\widetilde{\mathcal M}_{J_t,\beta}(\gamma,\gamma')$ of smooth maps $u: \R\times [0,1] \to M$ satisfying:
\begin{itemize}
    \item $\frac{\partial u}{\partial s} + J_t \left( \frac{\partial u}{\partial t} - X_H(u)\right)=0$ (where $(s,t)\in \R\times [0,1]$);
    \item $u$ has finite energy;
    \item $u(\mathbb R,0)\subset L$ and $u(\mathbb R,1)\subset K$;
    \item $[u]=\beta$;
    \item $u$ is asymptotic to $\gamma$ at $s=-\infty$, and to $\gamma'$ at $s=+\infty$.
\end{itemize}
Denote by $\mathcal M_{J_t,\beta}(\gamma,\gamma')$ the quotient of this moduli space by the action of $\R$ by translation, and by $\pi_2(\gamma,\gamma')$ the set of homotopy classes of Floer trajectories between $\gamma$ and $\gamma'$.
The differential is then given by 
\[
\partial [\gamma',w] := \sum\limits_{[\gamma,w\#\beta]\in Crit(\mathcal A_H)}\#\mathcal M_{J_t,\beta}(\gamma,\gamma')[\gamma,w\#\beta]
\]

\subsection{Identification of the vector spaces}\label{ss:identificationvsp}

Let $\underline L$ and $\underline K$ be two transverse Lagrangian links with $k$ components on a closed surface $(\Sigma, \omega)$. Let $(E,\omega_E,j_E)$ denote the two-dimension torus with complex structure $j_E$ and K\"ahler form $\omega_E$.
Let $\alpha$ be a non-contractible circle on $E$ and $\alpha '$ be a small Hamiltonian deformation of $\alpha$, such that $\alpha$ and $\alpha '$ are transverse. Let $\sigma_1\in\Sigma\setminus (\underline L\cup\underline K)$, and $\sigma_2$ be a point in $E$ away from the isotopy between $\alpha$ and $\alpha'$.

We denote by $\Sigma'(T)$ the connected sum of $\Sigma$ and $E$ along the points $\sigma_1$ and $\sigma_2$, that we construct in the following way:

Pick small real numbers $r_1$ and $r_2$, and fix conformal identifications $\Phi_1:B_{r_1}(\sigma_1)\setminus\{\sigma_1\}\xrightarrow{\simeq} [0,\infty)\times S^1$ and $\Phi_2:B_{r_2}(\sigma_2)\setminus\{\sigma_2\}\xrightarrow{\simeq} [0,\infty)\times S^1$ (where $B_r(z)$ denotes the closed ball of radius $r$ centered at $z$).
Let $\Sigma(2T):=\Sigma\setminus \Phi_1^{-1}\left((2T,\infty)\times S^1\right)$ and $E(2T):=E\setminus \Phi_2^{-1}\left((2T,\infty)\times S^1\right)$.
Then, $\Sigma'(T)$ is the union of $\Sigma(2T)$ and $E(2T)$ modulo the identification of the cylinders $[0,2T]\times S^1\subset\Sigma(2T)$ and $[0,2T]\times S^1\subset E(2T)$ via the involution $(t,\theta)\sim (2T-t,\theta)$.
We denote the resulting complex structure on $\Sigma'(T)$ by $j'(T)$, which agrees with $j$ over $\Sigma \setminus B_{r_1}(\sigma_1)$, agrees with $j_E$ over $E \setminus B_{r_2}(\sigma_2)$ and agrees with the standard complex structure over the tube $[0,2T]\times S^1$.

We assume the Hamiltonian isotopy from $\alpha$ to $\alpha'$ is small enough such that the area of its support is less than $\omega(B_{r_1}(\sigma_1))$.
In this case, we can equip $\Sigma'(T)$ with a symplectic form $\omega'(T)$ which agrees with $\omega$ over  $\Sigma \setminus B_{r_1}(\sigma_1)$, agrees with $\omega_E$ over the support of the Hamiltonian isotopy from $\alpha$ to $\alpha'$, is compatible with $j'(T)$, and $\omega'(T)(\Sigma'(T))=\omega(\Sigma)$.

Let $W:=\{\sigma_1\}\times \Sym^{k-1}(\Sigma)\subset \Sym^k(\Sigma)$.
Let $\sigma$ be a point that lies in the same connected component of $\Sigma\setminus (\underline L\cup\underline K)$ as $\sigma_1$, but away from $B_{r_1}(\sigma_1)$.

For any $z\in\Sigma\setminus(\underline L\cup \underline K)$ and  $\varphi\in H_2(\Sym^k(\Sigma),\Sym(\underline L)\cup\Sym(\underline K))$, we denote by $n_z(\varphi)$ the intersection number of $\varphi$ with $\{z\}\times \Sym^{k-1}(\Sigma)\subset \Sym^k(\Sigma)$.
Similarly, for $z'\in\Sigma'\setminus(\underline L\cup \underline K\cup\alpha\cup\alpha')$ and $\varphi'\in H_2(\Sym^{k+1}(\Sigma'),\Sym(\underline L\cup\alpha)\cup\Sym(\underline K\cup\alpha'))$, we denote by $n'_{z'}(\varphi')$ the intersection number of  $\varphi'$ with $\{z'\}\times \Sym^{k}(\Sigma')\subset \Sym^{k+1}(\Sigma')$ where .
For $z_E\in E\setminus(\alpha\cup\alpha')$, and $\varphi_E$ in $H_2(E,\alpha\cup\alpha')$, we denote by $n^E_{z_E}(\varphi_E)$ the intersection number of $\varphi_E$ with $z_E$.

In order to prove Theorem \ref{thm:stabilization}, we start by establishing an isomorphism of vector spaces between the Floer complexes. We will show that there is a one-to-one correspondence between generators.

Given an intersection point $x=\{x_1,...,x_k\}\in\Sym^k(\Sigma)$ between $\Sym \underline L$ and $\Sym \underline K$, and an intersection point $c\in E$ between $\alpha$ and $\alpha '$, we get an intersection point $x\times\{c\}\in\Sym^k(\Sigma\setminus B_{r_1}(\sigma_1))\times \Sym^1(E\setminus B_{r_2}(\sigma_2))\subset\Sym^{k+1}(\Sigma\#_T E)$ between $\Sym (\underline L\cup \alpha)$ and $\Sym (\underline K\cup\alpha ')$.

Since $\alpha$ and $\alpha '$ do not intersect $\underline L$ and $\underline K$, any intersection point between $\Sym (\underline L\cup \alpha)$ and $\Sym (\underline K\cup\alpha ')$ can be decomposed in a single way as $x\times\{c\}$ where $x\in\Sym(\underline L)\cap\Sym(\underline K)$ and $c\in\alpha\cap\alpha '$, and therefore there is a one-to-one correspondence between $\left(\Sym(\underline L)\cap\Sym\underline K)\right)\times(\alpha\cap\alpha')$ and $\Sym (\underline L\cup \alpha)\cap\Sym (\underline K\cup\alpha ')$.

Now we need to consider the cappings. In fact, recall that a generator of the Floer complex is an equivalence class of an intersection point $x$ together with a capping $w$ (cf. Definition \ref{d:equivalentCap}).

For each connected component of $\mathcal P(\Sym \underline L,\Sym \underline K)$ that contains an intersection point of $\Sym \underline L$ and $\Sym \underline K$, we choose the reference path $\eta$ for that connected component to be a constant path at one the intersection points that are contained in that component. 
On $E$, since we assumed that $\alpha'$ is a Hamiltonian perturbation of $\alpha$, we can assume that all intersection points between them are in the same connected component of $\mathcal P(\alpha,\alpha')$. 
We choose as a reference path $\eta_E$ to be a constant path equal to an intersection point between $\alpha$ and $\alpha'$. 
For every $\eta$ chosen above, We define $\eta':= \eta\times\eta_E$, and choose it as a reference path in $\Sym^{k+1}(\Sigma')$.
Clearly, the constant path at any intersection point between $\Sym (\underline L \cup \alpha)$ and $\Sym (\underline K \cup \alpha')$ is homotopic to some $\eta'$.
Moreover, if $\eta_1$ and $\eta_2$ are two reference paths in $\mathcal P(\Sym \underline L,\Sym \underline K)$ that are not in the same connected component, then $\eta_1'= \eta_1 \times\eta_E$ and $\eta_2'=\eta_2 \times \eta_E$ do not lie in the same connected component too.

\begin{definition}
    Let $\underline L$ and $\underline K$ be two Lagrangian links on $\Sigma$ away from a point $z$. Let $x$ be a path between $\Sym\underline L$ and $\Sym\underline K$.
    A class $\varphi$ in $\pi_2(x,x)$ is said to be periodic if $n_z(\varphi)=0$. The set of periodic classes will be denoted $\Pi_x(z)$.
\end{definition}

Then, we show the following lemma (which is a generalization of \cite[Proposition 2.15]{OS}, which corresponds to the case k=g):

\begin{lemma}
    Let $(\Sigma,z)$ be a pointed surface. Let $\underline L$ and $\underline K$ be two Lagrangian links on $\Sigma$, away from $z$, with $k$ components.
    Then, for any path $x$ from $\Sym\underline L$ to $\Sym\underline K$,
    \[\pi_2(x,x)\cong \pi_2(\Sym^k(\Sigma))\oplus \Pi_x(z)\]
\end{lemma}

\begin{proof}
    $\pi_2(x,x)$ is the fundamental group of $\mathcal P(\Sym\underline L,\Sym\underline K)$ based at the point $x$. The evaluation at both ends of the path gives rise to a fibration \[\Omega \Sym^k(\Sigma)\to\mathcal P(\Sym\underline L,\Sym\underline K)\to \Sym\underline L\times\Sym\underline K\]
    The corresponding long exact sequence gives
    \[0\to\pi_2(\Sym^k(\Sigma))\to\pi_1(\mathcal P(\Sym\underline L,\Sym\underline K))\to \pi_1(\Sym\underline L\times\Sym\underline K)\xrightarrow{f}\pi_1(\Sym^k(\Sigma))\]
    One can rewrite it as a short exact sequence
    \[0\to\pi_2(\Sym^k(\Sigma))\to\pi_2(x,x)\to\Ker(f)\to 0\]
    Since $k \ge g$, we have either $\pi_2(\Sym^k(\Sigma))\cong \Z$ or $\pi_2(\Sym^k(\Sigma))=0$, and $\pi_2(\Sym^k(\Sigma))=0$ happens only when $k=g=1$ (cf. \cite[Theorem 5.4]{BR14}).
    When $\pi_2(\Sym^k(\Sigma))\cong \Z$, the map $\pi_2(x,x)\xrightarrow{n_z}\Z$ gives a splitting of the short exact sequence, and $\pi_2(x,x)\cong \pi_2(\Sym^k(\Sigma))\oplus \Ker(f)$. When $\pi_2(\Sym^k(\Sigma))=0$, we also have $\pi_2(x,x)\cong \pi_2(\Sym^k(\Sigma))\oplus \Ker(f)$.
    Since $n_z$ gives a splitting of the sequence, $\Ker(f)$ can be identified with $\Ker(n_z)=\Pi_x(z)$, which shows that $\pi_2(x,x)\cong \pi_2(\Sym^k(\Sigma))\oplus \Pi_x(z)$.
    
\end{proof}

\begin{lemma}Given an intersection point $x=\{x_1,...,x_k\}\in\Sym^k(\Sigma)$ between $\Sym \underline L$ and $\Sym \underline K$, all cappings are of the form $[x,nS\#w]$ where $S$ is a generator of $\pi_2(\Sym^k(\Sigma))$ whose intersection number with $W$ is $n_{\sigma_1}(S)=1$, $w$ is a capping from $x$ to $\eta$ that does not intersect $W$, and $n$ is an integer.
\end{lemma}

\begin{proof}
    Fix a capping $w_x$ from $\eta$ to $x$ that does not intersect $W$. Then, $\pi_2(x,\eta)=w_x \# \pi_2(\eta,\eta)$.
    By the previous lemma, $\pi_2(\eta,\eta)\cong \pi_2(\Sym^k(\Sigma))\oplus \Pi_\eta(\sigma_1)$.
    Therefore,
    \[\pi_2(x,\eta)\cong \pi_2(\Sym^k(\Sigma))\oplus (w_x \# \Pi_\eta(\sigma_1))\]
    and elements of $\Pi_\eta(\sigma_1)\#w_x$ do not intersect $W$.
\end{proof}

Given an intersection point $c\in E$ between $\alpha$ and $\alpha '$, since $E$ is aspherical, all cappings from $c$ to $\eta_E$  do not intersect $\sigma_2$.
Moreover, since $\alpha$ and $\alpha '$ are Hamiltonian isotopic to each other, any two cappings would have the same areas and hence descend to a unique equivalence class.

\begin{lemma}\label{l:capId}
Given an intersection point $x\times\{c\}$ between $\Sym (\underline L\cup \alpha)$ and $\Sym (\underline K\cup\alpha ')$, a capping $w$ from $x$ to $\eta$ that does not intersect $W$, and $w_E$ from $c$ to $\eta_E$ that does not intersect $\sigma_2$, $w\times w_E$ is a capping from $x\times\{c\}$ to $\eta'$ such that $n'_\sigma(w\times w_E)=0$.
Moreover, all cappings from $x\times\{c\}$ to $\eta'$ are of the form $[x\times\{c\},nS'\#(w\times w_E)]$ for some integer $n$ and cappings $w$ and $w_E$ as above, and where $S'$ is a generator of $\pi_2(\Sym^{k+1}(\Sigma'))$ with $n'_\sigma(S') = 1$.
\end{lemma}

\begin{proof}
    The first part of the lemma is straightforward. The proof of the second part is identical to that of the previous lemma.
\end{proof}

\begin{prop}\label{p:capId}
    The linear map defined by
    \begin{align*}\Phi : CF^*(\Sym \underline L, \Sym\underline K)\otimes CF^*(\alpha,\alpha ') &\to CF^*(\Sym (\underline L\cup \alpha),\Sym (\underline K\cup\alpha '))\\
    [x,nS\#w]\otimes [c,w_E] &\mapsto [x\times\{c\},nS'\#(w\times w_E)]
    \end{align*}
    is an isomorphism of vector spaces.
\end{prop}

\begin{proof}
    We already know there is a one-to-one correspondence between $\left(\Sym(\underline L)\cap\Sym\underline K)\right)\times(\alpha\cap\alpha')$ and $\Sym (\underline L\cup \alpha)\cap\Sym (\underline K\cup\alpha ')$, and according to the previous lemma, the mapping $\Psi:\pi_2(x,\eta)\times\pi_2(c,\eta_E)\to\pi_2(x\times\{c\},\eta')$ is also a one-to-one correspondence. It remains to show that this mapping descends to a one-to-one correspondence between equivalence classes of cappings 
    \[\pi_2(x,\eta)/\Ker(\Sym(\omega))\times\pi_2(c,\eta_E)/\Ker(\omega_{E})\to\pi_2(x\times\{c\},\eta')/\Ker(\Sym(\omega'(T))).\]

First note that in the torus, $\pi_2(c,\eta_E)/\Ker(\omega_E)$ is trivial (see the paragraph before Lemma \ref{l:capId}).

Then recall that $\omega'(T)$ is chosen such that it agrees with $\omega$ over  $\Sigma \setminus B_{r_1}(\sigma_1)$, agrees with $\omega_E$ over the support of the Hamiltonian isotopy from $\alpha$ to $\alpha'$, is compatible with $j'(T)$ and $\omega'(T)(\Sigma'(T))=\omega(\Sigma)$.
    These conditions guarantee that for all $w$ and $w_E$, we have $\Sym(\omega'(T))(\Psi(w,w_E))=\Sym(\omega)(w)+\omega_E(w_E)$.
    It implies the result.
\end{proof}

To show that it is an isomorphism of chain complexes, we need to show that it preserves the differential, i.e. that for all such $x$ and $c$, $\partial(x\times\{c\})=\Phi((\partial x) \otimes c + x\otimes (\partial c))$\footnote{More precisely, we need to identify the differentials of the {\it capped} intersection points, but the identification of the cappings is straightforward so we focus on the intersection points.}.

In order to do this, we compare the moduli space $\mathcal M(x,y)$ of Maslov index $1$ Floer trajectories from $x$ to some $y$ in $\Sym^k(\Sigma)$ to $\mathcal M(x\times\{c\},y\times\{c\})$, and $\mathcal M(c,c')$ to $\mathcal M(x\times\{c\},x\times\{c'\})$.

In fact, we will show that these moduli spaces can be identified when considering a complex structure on $\Sigma '=\Sigma\#E$ that stretches sufficiently the connected sum tube.

This is a generalization of a statement in \cite{OS} which only considers the case of links with $g$ components (where $g$ is the genus of $\Sigma$), and circles $\alpha$ and $\alpha '$ with a single intersection point.
The proof of this statement still works in our setting. We will recall the main steps of this proof and emphasize where one have to be careful when generalizing.

Before discussing the moduli spaces of Floer trajectories, which are pseudo-holomorphic disks, we have to fix paths of almost complex structures on each manifold.

We fix
a path of $V$-nearly symmetric almost complex structures $(J_t)_{t \in [0,1]}$ on $\Sym^k(\Sigma)$, for some neighborhood $V$ of $\Delta \cup \Sym^{k-1}(\Sigma)\times\{\sigma_1\}\subset \Sym^k(\Sigma)$.

Recall that $j'(T)$ is the complex structure on $\Sigma'=\Sigma\#_T E$ that coincides with $j$ on $\Sigma\setminus B_{r_1}(\sigma_1)$, with $j_E$ on $E\setminus B_{r_2}(\sigma_2)$, and is the standard cylindrical complex structure on the connected sum tube $[0,2T]\times S^1$ between $\Sigma$ and $E$.

Then, the symmetric product $\Sym^{k+1}(\Sigma')$ endowed with the complex structure $\Sym^{k+1}(j'(T))$ admits an open subset holomorphically identified with
\[\Sym^k(\Sigma-B_{r_1}(\sigma_1))\times \Sym^1(E-B_{r_2}(\sigma_2))\]

Fix $R_1>r_1$ and $R_2>r_2$.
We choose a path of almost complex structures $J'_t(T)$ on $\Sym^{k+1}(\Sigma')$ satisfying the following conditions:

\begin{itemize}

    \item $J'_t(T)\equiv J_t\times j_E$ on $\Sym^k(\Sigma-B_{R_1}(\sigma_1))\times \Sym^1(E-B_{R_2}(\sigma_2))$
    
    \item $J'_t(T)=J_{t,r}\times j_E$ on $\Sym^k(\Sigma-B_{r_1}(\sigma_1))\times \Sym^1(B_{R_2}(\sigma_2)-B_{r_2}(\sigma_2))$, where $J_{t,r}$ is $V$-nearly symmetric for all $r$ and connects $J_t$ to $\Sym^k(j)$ as $r$, the normal parameter to $\sigma_2$, goes from $R_2$ to $r_2$. 
    
    \item $J'_t(T)\equiv \Sym^{k+1}(j'(T))$ on the rest of $\Sym^{k+1}(\Sigma')$ 
    
\end{itemize}

In particular, $J'_t(T)$ is $V'$-nearly symmetric for some neighborhood $V'$ of the diagonal $\Delta'\subset \Sym^{k+1}(\Sigma')$.

Let $x,y\in \Sym(\underline L)\cap\Sym(\underline K)$, and $c,c'\in\alpha\cap\alpha'$.

Given $\varphi\in\pi_2(x,y)$, there is a single class $\varphi'_c\in\pi_2(x\times\{c\},y\times\{c\})$ such that for any $z\in\Sigma\setminus(\underline L\cup \underline K)$, $n_z(\varphi)=n'_z(\varphi'_c)$.
Similarly, for any $\varphi_E\in\pi_2(c,c')$, there is a single class $\varphi'_{E,x}\in\pi_2(x\times\{c\},x\times\{c'\})$ such that for any $z_E\in E\setminus(\alpha\cup\alpha')$, $n^E_{z_E}(\varphi_E)=n'_{z_E}(\varphi'_{E,x})$.

Then, Theorem \ref{thm:stabilization} is a consequence of the following statement:

\begin{theorem}
    \label{thm:moduli}
    Let $\varphi\in\pi_2(x,y)$ and $\varphi_E\in\pi_2(c,c')$ be two classes of Maslov index $1$.
    Then, for sufficiently large $T$, $\mathcal M_{J_t,\varphi}(x,y)\simeq \mathcal M_{J'_t(T),\varphi'_c}(x\times\{c\},y\times\{c\})$ and $\mathcal M_{j_E,\varphi_E}(c,c')\simeq \mathcal M_{J'_t(T),\varphi'_{E,x}}(x\times\{c\},x\times\{c'\})$.
\end{theorem}

\begin{remark}
    The isomorphisms above are identifications between deformation theories, and therefore $\mu(\varphi'_c)=\mu(\varphi)=1$, and $\mu(\varphi'_{E,x})=\mu(\varphi_E)=1$.
\end{remark}

The proof of this theorem consists of two steps:
\begin{itemize}
    \item Given a pseudo-holomorphic disk in $\Sym^k(\Sigma)$, we construct a corresponding disk in $\Sym^{k+1}(\Sigma')$ by gluing spheres;
    \item By a Gromov compactness argument, we show that all Maslov index $1$ pseudo-holomorphic disks in $\Sym^{k+1}(\Sigma')$ can be constructed in this way.
\end{itemize}
These two steps will be addressed in the next two subsections respectively.

\subsection{Gluing}

Let $u$ be a pseudo-holomorphic disk in $\mathcal M_{j_E,\varphi_E}(c,c')$. Then, $u$ does not intersect $\sigma_2$, and therefore $x\times u$ defines a trajectory in $\Sym^{k+1}(\Sigma')$ between $x\times\{c\}$ and $x\times\{c'\}$, which is $J'_t(T)$-holomorphic. Moreover, for any $z_E\in E\setminus(\alpha\cup\alpha')$, $n'_{z_E}(u)=n^E_{z_E}(\varphi_E)$, so $x\times u\in\mathcal M_{J'_t(T),\varphi'_{E,x}}(x\times\{c\},x\times\{c'\})$.

Let $u$ be a pseudo-holomorphic disk in $\mathcal M_{J_t,\varphi}(x,y)$. When $n_{\sigma_1}(\varphi)=0$, we can construct $u':=u\times\{c\}$, and as before it lives in $\mathcal M_{J'_t(T),\varphi'_c}(x\times\{c\},y\times\{c\})$.

However when $n:=n_{\sigma_1}(\varphi)\neq 0$, we need to glue $n$ spheres to the disk $u$ to construct a disk in $\Sym^{k+1}(\Sigma')$. We follow the construction of \cite{OS}, which was done in the case $k=g$, but still works in this more general case. We will only give an outline of the proof without going into the more technical details, which are exactly the same as in \cite[Section 10.2 and 10.3]{OS}.

Suppose $u$ meets $W=\{\sigma_1\}\times \Sym^{k-1}(\Sigma)$ transversally in $n$ distinct points $q_1,...,q_n$.
We identify $\R \times [0,1]$ with $D\setminus\{\pm i\}$, where $D$ denotes the unit disk in $\C$.
Then, $u$ extends continuously to $D$ by setting $u(-i) = x, u(i)=y$.

We fix constants $0<r_1<R_1$ such that $u(D)\cap\left(B_{r_1}(\sigma_1)\times\Sym^{k-1}(\Sigma-B_{r_1}(\sigma_1))\right)\subset B_{r_1}(\sigma_1)\times\Sym^{k-1}(\Sigma-B_{R_1}(\sigma_1))$.

There exists $\epsilon>0$ such that for every $1\leqslant i\leqslant n$, $B_\epsilon(q_i)$ is mapped by $u$ into this subset.

We fix conformal identifications $B_{r_1}(\sigma_1)-\sigma_1\cong [0,\infty)\times S^1$, and $B_\epsilon(q_i)\cong [0,\infty)\times S^1$.

We will use Sobolev spaces with weight function $e^{\delta\tau_1}$, where:
\begin{itemize}
    \item $\delta$ is a positive constant;
    \item $\tau_1:D-\{q_1,...,q_n\}\to[0,\infty)$ is supported inside the $B_\epsilon(q_i)$;
    \item $\tau_1(s,\varphi) = s$ for $s\geqslant 1$ in each $B_\epsilon(q_i)\cong [0,\infty)\times S^1$.
\end{itemize}

Then, for each $i$ there exists $w_i\in \Sym^{k-1}(\Sigma)$, and $(t_i,\theta_i)\in \R\times S^1$ such that the restriction of $u$ to $B_\epsilon(q_i)-\{q_i\}\cong [0,\infty)\times S^1$ differs by a $L^p_{1,\delta}$ map from the smooth map 
\[a_{t_i,\theta_i,w_i}:[0,\infty)\times S^1\to \Sym^{k-1}(\Sigma)\times[0,\infty)\times S^1\subset \Sym^{k}(\Sigma)\]
defined by 
\[a_{t_i,\theta_i,w_i}(s,\varphi)=(w_i,s+t_i,\varphi+\theta_i)\]
where we cut-off $s+t_i$ if it is negative (\cite[Section 10.2]{OS}).

Given $T>0$, we define $X_1(T):=\tau_1^{-1}([0,T])$ and $X_1(\infty) = D-\{q_1,...,q_n\}$.

Let $h:\R\to[0,1]$ be a smooth, increasing function such that $h(t)=0$ for $t<0$ and $h(t)=1$ for $t>1$.

We can define a map $\widetilde u_T:X_1(\infty)\to\Sym^k(\Sigma)$ which agrees with $u$ away from the $B_\epsilon(q_i)$, and defined by
\[\widetilde u_T(s,\varphi)=h(s-T)a_{t_i,\theta_i,w_i}(s,\varphi)+(1-h(s-T))u(s,\varphi)\]
over $B_\epsilon(q_i) \setminus \{q_i\} \cong [0,\infty)\times S^1$, and extends smoothly over $q_i$ (where the convex combination is to be interpreted using the exponential map).

We also fix a constant $\delta_0>0$, and define $\tau_0:\R \times [0,1] \cong D\to\R$ supported away from the $B_\epsilon(q_i)$, and such that $\tau_0(s,t)=|s|$ for sufficiently large $s$\footnote{In \cite{OS}, they use $(s,t) \in [0,1] \times \R$ and require that $\tau_0$ equals to $|t|$ for sufficiently large $t$, but we follow the more standard convention that $(s,t) \in \R \times [0,1]$.}.

Then, according to \cite[Lemma 10.6]{OS}, for the Sobolev norm with weight $e^{\delta_0\tau_0+\delta\tau_1}$, there are constants $\kappa>0$, $S_0>0$ and $C>0$ such that for all $S>S_0$
\[\vv \bar{\partial}_{J_t}\widetilde u_S\vv_{L^p_{\delta,\delta_0}(\Lambda^{0,1}\widetilde u_S)}\leqslant Ce^{-\kappa S}\]

Now we consider spheres in $\Sym^2(E)$.
Let $v$ be a holomorphic map from $S^2$ to $\Sym^{k-1}(\Sigma)\times\Sym^2(E)$, constant on the first factor, and such that $n_{\sigma_2}([v])=1$. Denote by $\mathcal M(S^2\to\Sym^{k-1}(\Sigma)\times\Sym^2(E))$ the moduli space of such maps, modulo holomorphic reparametrization.
According to \cite[Lemma 10.7]{OS}, we have:
\begin{lemma}
    For such a map $v$, there exists a unique pair $(w,c)$ in $\Sym^{k-1}(\Sigma) \times E$  such that $(w,\{c,\sigma_2\})\in\im(v)$.\footnote{In \cite{OS}, they use the notation $(w,y)$ instead of $(w,c)$. We use $y$ to denote an intersection point between $\Sym\underline L$ and $\Sym \underline K$ so we use $c$ here.} 
    
    The map $[v]\mapsto(w,c)$ is then a one-to-one correspondence between $\mathcal M(S^2\to\Sym^{k-1}(\Sigma)\times\Sym^2(E))$ and $\Sym^{k-1}(\Sigma)\times E$. 
\end{lemma}
We fix $v$ as above, and normalize it so that $v(0)=w\times\{c,\sigma_2\}$ (where we view $S^2$ as $\C\cup\{\infty\}$). 

We will only be interested in the case that $c \in \alpha \cap \alpha'$.
The intuitive reason is that we are going to glue $u \times \{ c \}:D \to \Sym^k(\Sigma) \times E$ and $n$ many $v: S^2 \to \Sym^{k-1}(\Sigma) \times \Sym^2(E)$ together (one $v$ for each $q_i$), so $c$ has to be an intersection point between $\alpha$ and $\alpha'$ for $u \times \{c\}$ to satisfy the Lagrangian boundary conditions (cf. the degeneration \eqref{eq:Degeneration} where the Gromov limit lives).
Therefore, we assume from now on that $c\neq \sigma_2$. 

We identify a neighborhood of $v(0)$ with
\[\Sym^{k-1}(\Sigma)\times B_{r_2}(\sigma_2)\times(E-B_{R_2}(\sigma_2))\subset \Sym^{k-1}(\Sigma)\times\Sym^2(E) \]
for some $0<r_2<R_2$.

Fix $\epsilon>0$ such that $B_\epsilon(0)$ is sent by $v$ to this neighborhood.
Fix conformal identifications $[0,\infty)\times S^1\cong B_\epsilon(0)-\{0\}$ and $[0,\infty)\times S^1\cong B_{r_2}(\sigma_2)-\{\sigma_2\}$.

Then, there are unique $w\in \Sym^{k-1}(\Sigma)$, $c\in E$, $(t,\theta)\in[0,\infty)\times S^1$ such that $v$ restricted to $[0,\infty)\times S^1\cong B_\epsilon(0)-\{0\}$ differs by a $L^p_{1,\delta}$ map from the map
\[b_{(t,\theta,w,y)}(s,\varphi)=(w,s+t,\theta+\varphi,c)\]
(where $L^p_{1,\delta}$ is defined with a weight function $e^\delta\tau_2$ with $\tau_2:S^2-\{0\}\to\R^+$ defined in a similar fashion as $\tau_1$).

For $S>0$, let $S^2(S):=\tau_2^{-1}([0, S])$. We define a map\footnote{In \cite{OS}, the domain of $v_S$ is $S^2-\{\infty\}$ which we believe is a typo.}
\[v_S:S^2-\{0\}\to\Sym^{k-1}(\Sigma)\times\Sym^2(E-\{\sigma_2\})\]
which agrees with $v$ over $S^2(S)$, and such that over $[0,\infty)\times S^1\cong B_\epsilon(0)-\{0\}$,
\[v_S(s,\varphi)=h(s-S)b_{(t,\theta,w,c)}(s,\varphi)+(1-h(s-S))v(s,\varphi)\]

In \cite[Definition 10.8]{OS}, the authors define a normalization condition on holomorphic spheres called being \textit{'centered'}. They show that the moduli space of centered maps $\mathcal M^{\text{cent}}(S^2\to\Sym^{k-1}(\Sigma)\times\Sym^2(E))$ is diffeomorphic to $\Sym^{k-1}(\Sigma)\times\R\times S^1\times E$ through the assignment $v\mapsto (w,t,\theta,c)$.

Denote by $v_{(w,t,\theta,c)}$ the pre-image of $(w,t,\theta,c)$ by this diffeomorphism.

Using the conformal identifications $B_{r_1}(\sigma_1)-\sigma_1\cong [0,\infty)\times S^1$ and $B_{r_2}(\sigma_2)-\sigma_2\cong [0,\infty)\times S^1$, one can think of $\Sigma'(T)$ as the union of $\Sigma(2T)$ and $E(2T)$ modulo the identification of the cylinders $[0,2T]\times S^1\subset\Sigma(2T)$ and $[0,2T]\times S^1\subset E(2T)$ via the involution $(t,\theta)\sim (2T-t,\theta)$.

Let $X_2(T):= \bigsqcup\limits_{i=1}^n S^2(T)_i$ and $X_1\cup_T  X_2$ be the union of $X_1(T)$ and $X_2(T)$ glued at their common boundary. We have that $X_1\cup_T  X_2\cong D^2$.

There exists some constant $t>0$ such that for any real numbers $S$ and $T$ such that $0<S<T-t$, given the pseudo-holomorphic disk $u\in\mathcal M_{J_s,\varphi}(x,y)$ we fixed earlier, and the intersection point $c\in \alpha \cap \alpha' \subset E$, one can define a map
\[\hat{\gamma}_c(u,S,T):D\cong X_1\cup_T  X_2\to\Sym^{k+1}(\Sigma\#_T E)\]
which agrees with $\widetilde u_S\times\{c\}$ over $X_1(T)$ and with $v_{(w_i,-t_i,\theta_i,c),S}$ on $S^2(T)_i\subset X_2(T)$.

Following \cite[Lemma 10.9]{OS}, if $S$ is sufficiently large, then for large $T$ this map is smooth, and for an appropriate Sobolev norm there are some positive constants $C$ and $a$ such that 
\[\vv\bar\partial_{J'_t(T)}\hat{\gamma}_c(u,S,T)\vv\leqslant Ce^{-aS}\]

One can show (\cite[Proposition 10.12]{OS}) that when $T$ is sufficiently large, the linearization of $\bar\partial_{J'_t(T)}$ for the spliced map from $X_1\cup_T  X_2$ admits a right inverse whose norm is bounded independent of $T$.

Then, applying the inverse theorem function (\cite[Proposition 10.14]{OS}), there is an $\epsilon > 0$ such that for sufficiently large $T$, there is a unique holomorphic curve $\gamma_c(u)$ which lies in an $\epsilon$-neighborhood of $\hat{\gamma}_c(u,S,T)$ (measured in the appropriate Sobolev norm).

This $\gamma_c(u)$ lives in $\mathcal M_{J'_t(T),\varphi'_c}(x\times\{c\},y\times\{c\})$.

\subsection{Gromov compactness}

Now we need to show that every map in $\mathcal M_{J'_t(T),\varphi'_c}(x\times\{c\},y\times\{c\})$ and $\mathcal M_{J'_t(T),\varphi'_{E,x}}(x\times\{c\},x\times\{c'\})$ can be attained by the construction of the previous section.
Once again, the argument is similar to the one in \cite{OS}. 

Let $x',y'\in\Sym^{k+1}(\Sigma')$ be two critical points of the action functional (i.e. intersection points of the Lagrangians $\Sym(\underline{L} \cup \alpha) \cap \Sym(\underline{K } \cup \alpha')$). Let $\varphi'$ be a Maslov index $1$ class in $\pi_2(x',y')$.

Then, according to \cite[Proposition 10.15]{OS}, any sequence $u_T\in\mathcal M_{J'_t(T),\varphi'}(x',y')$ with $T$ going to infinity has, up to passing to a subsequence, a Gromov limit $u_{\infty}$ mapping to 
\begin{align}\label{eq:Degeneration}
\Sym^{k+1}(\Sigma\vee E)= \bigcup\limits_{i=0}^{k+1}\Sym^{k+1-i}(\Sigma)\times\Sym^i(E).
\end{align}
We can think of the wedge sum $\Sigma\vee E$ as the degeneration of the connected sum $\Sigma\#_T E$ when the neck length $T$ goes to infinity.
The limit $u_{\infty}$ is analyzed in Lemma \ref{l:case1} and \ref{l:case2} below.

\begin{remark}
   An alternative way to think about this Gromov compactness is to consider the relative Hilbert scheme $\operatorname{Hilb}^{k+1}(\pi)$ of a Lefschetz fibration $\pi:E \to D$ over a disk $D$, where generic fibres are smooth and the singular fibre is $\Sigma\vee E$. The relative Hilbert scheme is smooth \cite[Proposition 3.7]{Perutz} and one can equip $\operatorname{Hilb}^{k+1}(\pi)$ with a one-parameter family of almost complex structures such that the projection to $D$ are pseudo-holomorphic, they are fibrewise $V$-nearly symmetric, and they agree with $J'_t(T)$ over some fibres such that $T \to \infty$ corresponds to degenerating to the central fibre.
   The central fibre of $\operatorname{Hilb}^{k+1}(\pi)$ has a canonical `cycle map' to $\Sym^{k+1}(\Sigma\vee E)$ (cf. \cite[Section 1.5.1]{Perutz}) and $u_{\infty}$ is the same as applying Gromov compactness inside $\operatorname{Hilb}^{k+1}(\pi)$ and then applying the cycle map to $\Sym^{k+1}(\Sigma\vee E)$.
\end{remark}

\begin{lemma}[cf. Proposition 10.16 of \cite{OS}]\label{l:case2}
If $(x',y')$ is of the form $(x\times\{c\},y\times\{c\})$, then $u_{\infty}$ consists of a main component of the form $u\times \{c\}$, where $u$ is a Maslov index 1 trajectory from $x$ to $y$ in $\Sym^k(\Sigma)$, together with possibly some sphere components of the form $\{w\}\times v$, where $w \in \Sym^{k-1}(\Sigma)$ and $v$ is a holomorphic sphere in $\Sym^2(E)$ with Chern number $2$ (i.e. a sphere in the ruling).
\end{lemma}

\begin{proof}
Since the $u_T$ are Floer trajectories between $x'$ and $y'$, $u_{\infty}$ consists of a (possibly) broken Floer trajectory
between $x'$ and $y'$, as well as sphere bubbles and disk bubbles.
Let $u_{\infty}^i$ be the components of $u_{\infty}$ in $\Sym^{k+1-i}(\Sigma)\times\Sym^i(E)$.
By projection to factors, we can write it as $u_{\infty}^i=(u_{\infty}^{\Sigma,i},u_{\infty}^{E,i})$.

Let $D_{\Sigma,i}=\Sym^{k-i}(\Sigma) \times \{\sigma_1\} \subset \Sym^{k+1-i}(\Sigma)$ and $D_{E,i}=\Sym^{i-1}(E) \times \{\sigma_2\} \subset \Sym^i(E)$.
We define the adjusted Maslov index $\tilde{\mu}(u_{\infty}^i)$ of $u_{\infty}^i$ relative to $D_{\Sigma,i}\times\Sym^i(E) + \Sym^{k+1-i}(\Sigma)\times D_{E,i}$ as the Maslov index of $u$ with respect to the log canonical line bundle with a simple pole along $D_{\Sigma,i}\times\Sym^i(E) + \Sym^{k+1-i}(\Sigma)\times D_{E,i}$.
In other words, the adjusted Maslov index of $u_{\infty}^i$ is its usual Maslov index viewed as a map to $\Sym^{k+1-i}(\Sigma)\times\Sym^i(E)$ subtracted by $2 [u_{\infty}^i] \cdot (D_{\Sigma,i}\times\Sym^i(E) + \Sym^{k+1-i}(\Sigma)\times D_{E,i})$.
The additivity of Maslov index under passing to a Gromov limit implies that the sum of the adjusted Maslov indices of the components of $u_{\infty}$ equals the Maslov index of $\varphi'$, which is $1$.
This is because $D_{\Sigma,i}\times\Sym^i(E) + \Sym^{k+1-i}(\Sigma)\times D_{E,i}$ is precisely the locus where $\Sym^{k+1-i}(\Sigma)\times\Sym^i(E)$ intersects other irreducible components of $\Sym^{k+1}(\Sigma\vee E)$.

Since $\Sym(\underline{L} \cup \alpha)$ and $\Sym(\underline{K } \cup \alpha')$ are contained in $\Sym^{k}(\Sigma)\times E\subset \Sym^{k+1}(\Sigma\vee E)$, the broken Floer trajectory and disk bubbles are contained in $\Sym^{k}(\Sigma)\times E$.
We denote the spherical components of $u_{\infty}^1$ by $u_{\infty,S}^1$ and the other components of $u_{\infty}^1$ by $u_{\infty,D}^1$.

Now, we analyze the adjusted Maslov indices of the spherical components of $u_{\infty}$.
Recall from \cite[Theorem 9.2]{BT01} that the rank of $\pi_2(\Sym^j(\Sigma)) \to H_2(\Sym^j(\Sigma))$ is $1$ when $j \ge 2$ or $\Sigma$ has genus $0$, and $0$ otherwise. 
Moreover, the Chern number of a spherical class $u$ is given by $(j+1-g)[u] \cdot [\Sym^{j-1}(\Sigma) \times \{\sigma_1\}]$ (cf. \cite[Remark 4.18]{CGHMSS1}).
Therefore, its adjusted Maslov index relative to $\Sym^{j-1}(\Sigma) \times \{\sigma_1\}$ is given by
\begin{align}\label{eq:MaslovN}
2c_1 \cdot [u] -2 [u] \cdot [\Sym^{j-1}(\Sigma) \times \{\sigma_1\}]=2(j-g)[u] \cdot [\Sym^{j-1}(\Sigma) \times \{\sigma_1\}]
\end{align}
The adjusted Maslov index $\tilde{\mu}(u_{\infty}^i)$ is the sum of the adjusted Maslov indices $\tilde{\mu}(u_{\infty}^{\Sigma,i})+\tilde{\mu}(u_{\infty}^{E,i})$ where the two terms in the sum are relative to $D_{\Sigma,i}\times\Sym^i(E)$ and $\Sym^{k+1-i}(\Sigma)\times D_{E,i}$, respectively. For spherical components, they can be computed by the formula \eqref{eq:MaslovN}.

For $i \neq 1$, we define 
\[
N_{\Sigma,i}:=[u_{\infty}^{\Sigma,i}] \cdot D_{\Sigma,i}, \quad N_{E,i}:=[u_{\infty}^{E,i}] \cdot D_{E,i}
\]
For $i=1$, we define
\[
N_{\Sigma,1}:=[u_{\infty,S}^{\Sigma,1}] \cdot D_{\Sigma,1}, \quad N_{E,1}:=[u_{\infty,S}^{E,1}] \cdot D_{E,1}
\]
and 
\[
P_{\Sigma,1}:=[u_{\infty,D}^{\Sigma,1}] \cdot D_{\Sigma,1}, \quad P_{E,1}:=[u_{\infty,D}^{E,1}] \cdot D_{E,1}
\]
The terms $P_{\Sigma,1}$ and $P_{E,1}$ make sense because the Lagrangian boundary condition splits as a product, and they are disjoint from the divisor $D_{\Sigma,1}$, $D_{E,1}$.
Clearly, $N_{E,0}=N_{E,1}=0$ and $N_{\Sigma,k+1}=0$ because the spherical class is trivial.
Note that $D_{\Sigma,i} \times \Sym^i(E) \subset \Sym^{k+1-i}(\Sigma)\times\Sym^i(E)$ and $\Sym^{k-i}(\Sigma) \times D_{E,i+1} \subset \Sym^{k-i}(\Sigma)\times\Sym^{i+1}(E)$ are precisely the locus where these two components of $\Sym^{k+1}(\Sigma\vee E)$ meet each other.
Therefore, we have
\begin{align}
    N_{\Sigma,i}=N_{E,i+1}
\end{align}
for $i \ge 2$, and  
\[
N_{\Sigma,1}+P_{\Sigma,1}=N_{E,2}, \quad N_{\Sigma,0}=N_{E,1}+P_{E,1}=P_{E,1}
\]

Recall that $\alpha'$ is a Hamiltonian push-off of $\alpha$ such that the Hamiltonian isotopy from $\alpha$ to $\alpha'$ does not pass through $\sigma_2$.
Therefore, any Floer trajectory between two intersection points of $\alpha$ and $\alpha'$ does not pass through $\sigma_2$ in $E$.
Also, there is no non-constant disk bubble in $E$ so $P_{E,1}=0$.
Therefore, $N_{\Sigma,0}=0$ and $u_{\infty}$ does not intersect the component $\Sym^{k+1}(\Sigma) \subset \Sym^{k+1}(\Sigma\vee E)$ at all.

The total sum of the adjusted Maslov indices of the components of $u_{\infty}$ is given by
\begin{align*}
&\tilde{\mu}(u_{\infty}^1)+\sum_{i=2}^{k+1} \tilde{\mu}(u_{\infty}^i) \\
=&\tilde{\mu}(u_{\infty,D}^1)+\tilde{\mu}(u_{\infty,S}^{\Sigma,1})+ \tilde{\mu}(u_{\infty,S}^{E,1}) +\sum_{i=2}^{k+1} \tilde{\mu}(u_{\infty}^{\Sigma,i}) +\sum_{i=2}^{k+1} \tilde{\mu}(u_{\infty}^{E,i}) \\
=&(\mu(u_{\infty,D}^1)-2P_{\Sigma,1})+2(k-g)N_{\Sigma,1}+0+\sum_{i=2}^{k+1} 2(k+1-i-g)N_{\Sigma,i} +\sum_{i=2}^{k+1} 2(i-1)N_{E,i}\\
=&\mu(u_{\infty,D}^1)-2P_{\Sigma,1}+2(k-g)N_{\Sigma,1}+\sum_{i=2}^{k} 2(k+1-i-g)N_{\Sigma,i} +2N_{E,2}+\sum_{i=2}^{k} 2iN_{E,i+1}\\
=&\mu(u_{\infty,D}^1)-2P_{\Sigma,1}+2(k-g)N_{\Sigma,1}+2(N_{\Sigma,1}+P_{\Sigma,1})+\sum_{i=2}^{k} 2(k+1-g)N_{\Sigma,i}\\
=&\mu(u_{\infty,D}^1)+2(k-g+1)N_{\Sigma,1}+\sum_{i=2}^{k} 2(k+1-g)N_{\Sigma,i}
\end{align*}
As we said earlier, this total sum has to be $1$.
By regularity, each Floer trajectory component of $u_{\infty,D}^1$ contributes at least $1$ to the Maslov index.
Any non-constant disk bubble in $u_{\infty,D}^1$ contributes at least $2$ to the Maslov index.
Since $k \ge g$, we have $(k-g+1)>0$.
Therefore, the sum is $1$ only if $N_{\Sigma,i}=0$ for all $i$, and $u_{\infty,D}^1$ consists of a single component.
It implies that $P_{\Sigma,1}=N_{E,2}$.

By genericity, we can assume that the component $u_{\infty,D}^1$ intersects $D_{\Sigma,1} \times E$ transversely.
Therefore, $u_{\infty}^2$ intersects $\Sym^{k-1}(\Sigma) \times D_{E,1}$ transversely.
Note that every component of $u_{\infty}^2$ projects to a constant in $\Sym^{k-1}(\Sigma)$ because $N_{\Sigma,2}=0$.
Since the domain has genus $0$, the bubbling is modeled on a tree and hence no component of $u_{\infty}^2$ can be a multiple cover of an underlying holomorphic sphere.
It implies that the sphere components of $u_{\infty}$ are of the form $\{w\}\times v$, where $w \in \Sym^{k-1}(\Sigma)$ and $v$ is a holomorphic sphere in $\Sym^2(E)$ with Chern number $2$.
The Floer trajectory component of $u_{\infty}$ is $u_{\infty}^1$, which goes between $x\times\{c\}$ and $y\times\{c\})$.
Therefore, it is of the form $u\times \{c\}$, where $u$ is a Maslov index 1 trajectory from $x$ to $y$ in $\Sym^k(\Sigma)$.

\end{proof}

\begin{remark}
The bubbling analysis in the proof of Lemma \ref{l:case2} provides the details of the phrase `by a dimension count' in the proof of \cite[Proposition 10.16]{OS} and at the same time confirms that we can generalize it to all $k \ge g$.
\end{remark}

\begin{lemma}\label{l:case1}
If $(x',y')$ is of the form $(x\times\{c\},x\times\{c'\})$, then $u_{\infty}$ consists of a single component and it is of a product form $\{x\}\times u$ where $u$ is a Maslov index 1 trajectory from $c$ to $c'$ in $E$.
\end{lemma}

\begin{proof}
The proof is the same as Lemma \ref{l:case2}. The only difference is that when $(x',y')$ is of the form $(x\times\{c\},x\times\{c'\})$, the Floer trajectory $u_{\infty,D}^1$ is of the form $\{x\}\times u$ where $u$ is a Maslov index 1 trajectory from $c$ to $c'$.
Therefore, we have $P_{\Sigma,1}=0$ and hence there is no sphere bubbles.




\end{proof}


Now we show that for sufficiently large $T$, any map in $\mathcal M_{J'_t(T),\varphi'}(x\times\{c\},y\times\{c\})$ is attained by the construction of the previous section.

We proceed by contradiction: suppose that there is a sequence $(T_m)$ going to infinity, and a sequence of disks $u_{T_m}\in\mathcal M_{J'_s(T_m),\varphi'}(x\times \{c\},y\times\{c\})$ that are not attained by the gluing construction.

By what precedes, we can extract a subsequence converging to a bubbletree $u_\infty$, consisting of a disk $u\times\{c\}$ and $n$ spheres $\{w_i\}\times v_i$.

Then, the authors of \cite{OS} show that for sufficiently large $m$, $u_{T_m}$ is in an $\epsilon$-neighborhood of the nearly holomorphic map $\hat\gamma_c(u,S,T_m)$ for some $0<S<T_m-t$ (for the suitable Sobolev distance).

But we showed in the previous section that there was a single $J'_t(T)$ holomorphic curve in such a neighborhood, namely the curve $\gamma_c(u,T_m)$. Therefore, for large $m$, $u_{T_m} = \gamma_c(u,T_m)$, which contradicts our assumption.

Hence by contradiction for large $T$, we have $\mathcal M_{J_s,\varphi}(x,y)\simeq \mathcal M_{J'_t(T),\varphi'_c}(x\times\{c\},y\times\{c\})$.

This concludes the proof of Theorem \ref{thm:moduli}.



\begin{thebibliography}{CGHM{\etalchar{+}}22}

\bibitem[Ban97]{Banyaga}
Augustin Banyaga.
\newblock The structure of classical diffeomorphism groups.
\newblock 1997.

\bibitem[BRa14]{BR14}
Marcel B\"{o}kstedt and Nuno~M. Rom\~{a}o.
\newblock On the curvature of vortex moduli spaces.
\newblock {\em Math. Z.}, 277(1-2):549--573, 2014.

\bibitem[BT01]{BT01}
Aaron Bertram and Michael Thaddeus.
\newblock On the quantum cohomology of a symmetric product of an algebraic
  curve.
\newblock {\em Duke Math. J.}, 108(2):329--362, 2001.

\bibitem[Buh22]{Buhovsky}
Lev Buhovsky.
\newblock On two remarkable groups of area-preserving homeomorphisms.
\newblock {\em arXiv:2204.08020}, 2022.

\bibitem[CGHM{\etalchar{+}}21]{CGHMSS1}
Daniel Cristofaro-Gardiner, Vincent Humili{\`e}re, Cheuk~Yu Mak, Sobhan
  Seyfaddini, and Ivan Smith.
\newblock Quantitative heegaard floer cohomology and the calabi invariant.
\newblock {\em arXiv:2105.11026}, 2021.

\bibitem[CGHM{\etalchar{+}}22]{CGHMSS2}
Dan Cristofaro-Gardiner, Vincent Humili{\`e}re, Cheuk~Yu Mak, Sobhan
  Seyfaddini, and Ivan Smith.
\newblock {Subleading asymptotics of link spectral invariants and homeomorphism
  groups of surfaces}.
\newblock {\em arXiv:2206.10749}, June 2022.

\bibitem[CGHR15]{CGHR}
Daniel Cristofaro-Gardiner, Michael Hutchings, and Vinicius Gripp~Barros Ramos.
\newblock The asymptotics of {ECH} capacities.
\newblock {\em Invent. Math.}, 199(1):187--214, 2015.

\bibitem[CGHS20]{CGHS20}
Daniel Cristofaro-Gardiner, Vincent Humili{\`e}re, and Sobhan Seyfaddini.
\newblock Proof of the simplicity conjecture.
\newblock {\em arXiv:2001.01792}, 2020.

\bibitem[Che21]{Chen21}
Guanheng Chen.
\newblock {Closed-open morphisms on periodic Floer homology}.
\newblock {\em arXiv:2111.11891}, 2021.

\bibitem[Che22]{Chen22}
Guanheng Chen.
\newblock {On PFH and HF spectral invariants}.
\newblock {\em arXiv:2209.11071}, 2022.

\bibitem[EP]{EP}
Michael Entov and Leonid Polterovich.
\newblock Calabi quasimorphism and quantum homology.
\newblock {\em Intern. Math. Res. Notices}, pages 1635--1676.

\bibitem[Fat80]{Fa80}
A.~Fathi.
\newblock Structure of the group of homeomorphisms preserving a good measure on
  a compact manifold.
\newblock {\em Ann. Sci. \'{E}cole Norm. Sup. (4)}, 13(1):45--93, 1980.

\bibitem[Hum17]{H}
Vincent Humilière.
\newblock Géométrie symplectique {$C^0$} et sélecteurs d'action, 2017.

\bibitem[Hut11]{Hutchings}
Michael Hutchings.
\newblock Quantitative embedded contact homology.
\newblock {\em J. Differential Geom.}, 88(2):231--266, 2011.

\bibitem[Lip06]{Lipshitz}
Robert Lipshitz.
\newblock A cylindrical reformulation of {H}eegaard {F}loer homology.
\newblock {\em Geometry {\&} Topology}, 10(2):955--1096, aug 2006.

\bibitem[LZ18]{LZ}
R\'{e}mi Leclercq and Frol Zapolsky.
\newblock Spectral invariants for monotone {L}agrangians.
\newblock {\em J. Topol. Anal.}, 10(3):627--700, 2018.

\bibitem[MS21]{MS21}
Cheuk~Yu Mak and Ivan Smith.
\newblock Non-displaceable {L}agrangian links in four-manifolds.
\newblock {\em Geom. Funct. Anal.}, 31(2):438--481, 2021.

\bibitem[OM07]{OM}
Yong-Geun Oh and Stefan M\"{u}ller.
\newblock The group of {H}amiltonian homeomorphisms and {$C^0$}-symplectic
  topology.
\newblock {\em J. Symplectic Geom.}, 5(2):167--219, 2007.

\bibitem[OS04]{OS}
Peter Ozsv\'{a}th and Zolt\'{a}n Szab\'{o}.
\newblock Holomorphic disks and topological invariants for closed
  three-manifolds.
\newblock {\em Ann. of Math. (2)}, 159(3):1027--1158, 2004.

\bibitem[OS08]{OS2}
Peter Ozsv{\'{a}}th and Zolt{\'{a}}n Szab{\'{o}}.
\newblock Holomorphic disks, link invariants and the multi-variable alexander
  polynomial.
\newblock {\em Algebraic {\&} Geometric Topology}, 8(2):615--692, may 2008.

\bibitem[Per07]{Perutz}
Tim Perutz.
\newblock Lagrangian matching invariants for fibred four-manifolds. {I}.
\newblock {\em Geom. Topol.}, 11:759--828, 2007.

\bibitem[PS21]{PS21}
Leonid Polterovich and Egor Shelukhin.
\newblock Lagrangian configurations and {H}amiltonian maps.
\newblock {\em arXiv:2102.06118}, 2021.

\bibitem[Sey13]{S}
Sobhan Seyfaddini.
\newblock {$C^0$}-limits of {H}amiltonian paths and the {O}h-{S}chwarz spectral
  invariants.
\newblock {\em Int. Math. Res. Not. IMRN}, (21):4920--4960, 2013.

\end{thebibliography}

\newcommand{\etalchar}[1]{$^{#1}$}
\def\cprime{$'$}

\end{document}